\documentclass{article}%
\usepackage{amsfonts}
\usepackage{amsmath}
\usepackage{geometry}
\usepackage{amssymb}
\usepackage{graphicx}
\usepackage{url}
\usepackage{hyperref}
\usepackage{accents}
\usepackage{afterpage}
\usepackage{endnotes}
\usepackage{lastpage}%
\providecommand{\U}[1]{\protect\rule{.1in}{.1in}}
\newtheorem{theorem}{Theorem}[section]
\newtheorem{lemma}[theorem]{Lemma}
\newtheorem{definition}[theorem]{Definition}
\newtheorem{example}[theorem]{Example}

\newtheorem{remark}[theorem]{Remark}
\numberwithin{equation}{section}

\newtheorem{claim}{Claim}

\newtheorem{corollary}[theorem]{Corollary}

\newtheorem{notation}{Notation}

\newtheorem{proposition}[theorem]{Proposition}

\newenvironment{proof}[1][Proof]{\noindent\textbf{#1.} }{\ \rule{0.5em}{0.5em}}
\begin{document}

\title{Rank two filtered $(\varphi,N)$-modules
with Galois descent data and coefficients}
\author{Gerasimos Dousmanis\\M\"{u}nster Universit\"{a}t, SFB 478 Geometrische Strukturen in der\\Mathematik, Hittorfstra\ss e 27, 48149 M\"{u}nster Deutschland\\Email:
{\href{mailto:makis.dousmanis@math.uni-muenster.de}{makis.dousmanis@math.uni-muenster.de}%
}}
\maketitle

\begin{abstract}
Let $K$ be any finite extension of $%
\mathbb{Q}
_{p},$ $L$ any finite Galois extension of $K$, and $E$ any finite large
enough coefficient field containing $L.$ We classify two-dimensional $L$%
-semistable $E$-representations of $G_{K}$, by listing the isomorphism
classes of rank two weakly admissible filtered $(\varphi ,N,L/K,E)$-modules.
\end{abstract}

\section{Introduction}

\noindent Let $K$ be any finite extension of $\mathcal{%
\mathbb{Q}
}_{p}$ and $\rho:G_{K}\rightarrow GL_{n}(\bar{%
\mathbb{Q}%
}_{p})$ any continuous $n$-dimensional representation of $G_{K}=$ Gal$(\bar{%
\mathbb{Q}%
}_{p}/K).$ Let $L$ be any finite Galois extension of $K.$ The representation
$\rho$ is called $L$-semistable if it becomes semistable when restricted to
$G_{L}.$ The field of definition $E$ of $\rho$ is a finite extension of
$\mathcal{%
\mathbb{Q}
}_{p}$ which may be extended to contain $L.$ Let $k\geq1$ be any integer. By a
variant of fundamental work of Colmez and Fontaine (\cite{CF00}), the category
of $L$-semistable $E$-representations of $G_{K}$ with Hodge-Tate weights in
the range $\{0,1,...,k-1\}$ is equivalent to the category of weakly admissible
filtered $\left(  \varphi,N,L/K,E\right)  $-modules $D\ $%
(Def.\ \ref{the filtered definition}),\ such that Fil$^{0}(L\otimes_{L_{0}%
}D)=L\otimes_{L_{0}}D\ $and Fil$^{k}(L\otimes_{L_{0}}D)=0.$ We classify
two-dimensional $L$-semistable $E$-representations of $G_{K},$ by listing the
isomorphism classes of rank two weakly admissible filtered $\left(
\varphi,N,L/K,E\right)  $-modules.

When $K\neq\mathcal{%
\mathbb{Q}
}_{p}$ interesting new phenomena occur, for example there exist disjoint
infinite families of irreducible two-dimensional crystalline representations
of $G_{K},$ sharing the same characteristic polynomial and filtration (Cor.
\ref{a}). Such families have been constructed in \cite{DO08} and their
semisimplified modulo $p$ reductions have been computed in \cite{DO09}.

Potentially semistable representations arise naturally in geometry. Deciding
which isomorphism classes of filtered modules occur from certain geometric
objects, e.g. Hilbert modular forms is an interesting open problem and we hope
that this paper will contribute in this direction. Special cases of the
problem have been treated by Fontaine and Mazur \cite{FM95} when both $E\ $and
$K\ $equal$\ \mathcal{%
\mathbb{Q}
}_{p}$ and $p\geq5,$ Breuil and M\'{e}zard \cite{BM02} who initiated the
subject with arbitrary coefficients, Savitt \cite{SAV05} in cases where the
representation becomes crystalline over tamely ramified extensions of
$\mathcal{%
\mathbb{Q}
}_{p},$ and most recently by Ghate and M\'{e}zard \cite{GM09} who treated
almost all cases where $K=\mathcal{%
\mathbb{Q}
}_{p},$ assuming that $E$ is large enough and $p\neq2.$ In this paper we
assume that the coefficient field $E$ is large enough, and make no further
assumptions. The paper is organized as follows: in the rest of this
introductory section we recall standard facts from $p$-adic Hodge theory and
there is nothing original. In Section \ref{product ring} we set up our main
notations and prove a canonical form lemma for Frobenius and the monodromy
operator (\S \ref{canonical forms}). We then proceed to determine the Galois
descent data (\S \ref{gal descent}). In Section \ref{stable filtrations} we
construct the Galois-stable filtrations and in Section
\ref{the fixed submodules} we compute Hodge and Newton invariants. In Section
\ref{list of the weakly admissible} we provide the complete list of rank two
weakly admissible filtered $\left(  \varphi,N,L/K,E\right)  $-modules,
determine which are irreducible, non-split reducible or split-reducible, and
describe their precise submodule structure. In Section
\ref{isomorphism classes} we list the isomorphism classes of rank two filters
modules (\S \ref{last list}), and in Section \ref{crystalline example section}
we apply the results of previous sections to explore new phenomena occurring
in the $K\neq%
\mathbb{Q}
_{p}$ case, focusing on crystalline representations.

\tableofcontents

\subsection{Fontaine's rings\label{fontaine's rings}}

Let $\mathbb{C}_{p}$ be the completion of $\bar{%
\mathbb{Q}%
}_{p}$ for the $p$-adic topology. The field $%
\mathbb{C}
_{p}$ is algebraically closed and complete. Let $\widetilde{E}=\underset
{x\mapsto x^{p}}{{\varprojlim}}\mathbb{C}_{p}=\{(x^{(0)},x^{(1)}%
,...,x^{(n)},...)\ $such\ that$\ (x^{(n+1)})^{p}=x^{(n)}\ for\ $%
all$\ n\geq0\}$ and let $\widetilde{E}^{+}$ be the set of $x=(x^{(0)}%
,x^{(1)},...,$\noindent$x^{(n)},...)\in\widetilde{E}$ with $v_{E}%
(x):=v_{p}(x^{(0)})\geq0.$ Then $\widetilde{E}$ with addition and
multiplication defined by
\[
(x+y)^{(n)}=\underset{m\rightarrow\infty}{{\lim}}(x^{(n+m)}+y^{(n+m)})^{p^{m}%
}\ \text{and }(xy)^{(n)}=x^{(n)}y^{(n)}%
\]
for\ all\ $n\geq0$ is an algebraically closed field of characteristic $p\ $and
$v_{E}$ is a valuation on $\widetilde{E}$ for which $\widetilde{E}$ is
complete with valuation ring $\widetilde{E}^{+}.$ Let $\widetilde{\mathbb{A}%
}^{+}$ be the ring of Witt vectors with $\widetilde{E}^{+}$-coefficients and
let $\widetilde{\mathbb{B}}^{+}=\widetilde{\mathbb{A}}^{+}[\frac{1}{p}%
]=\{\sum\limits_{k\gg-\infty}p^{k}[x_{k}],\ x_{k}\in\widetilde{E}^{+}\},$
where $[x]\in\widetilde{\mathbb{A}}^{+}$ is the Teichm\"{u}ller lift of
$x\in\widetilde{E}^{+}.$ The ring $\widetilde{\mathbb{B}}^{+}$ is endowed with
a ring epimorphism $\theta:$ $\widetilde{\mathbb{B}}^{+}\rightarrow
\mathbb{C}_{p}$ given by $\theta(\sum\limits_{k\gg-\infty}p^{k}[x_{k}%
])=\sum\limits_{k\gg-\infty}p^{k}x_{k}^{(0)}.$ By functorial properties of
Witt vectors the absolute Frobenius $\varphi:$ $\widetilde{E}^{+}%
\rightarrow\widetilde{E}^{+}$ lifts to a ring epimorphism $\varphi:$
$\widetilde{\mathbb{B}}^{+}\rightarrow\widetilde{\mathbb{B}}^{+}$given by
$\varphi(\sum\limits_{k\gg-\infty}p^{k}[x_{k}])=\sum\limits_{k\gg-\infty}%
p^{k}[x_{k}^{p}].$ Let $\varepsilon=(\varepsilon^{(i)})_{i\geq0}\in
\widetilde{E}$ where $\varepsilon^{(0)}=1$ and $\varepsilon^{(i)}$ is a
primitive $p^{i}$-th root of $1\ $such that $\left(  \varepsilon
^{(i+1)}\right)  ^{p}=\varepsilon^{(i)}\ $for all $i.$ If $\pi=[\varepsilon
]-1$ and $\pi_{1}=[\varepsilon^{\frac{1}{p}}]-1$, we write $\omega=\frac{\pi
}{\pi_{1}}.$ The kernel of the epimorphism $\theta:$ $\widetilde{\mathbb{B}%
}^{+}\rightarrow\mathbb{C}_{p}$ is the principal ideal generated by $\omega.$
The ring $\mathbb{B}_{dR}^{+}$ is defined to be the separated $\ker\theta
$-adic completion of $\widetilde{\mathbb{B}}^{+},\ $i.e. $\mathbb{B}_{dR}%
^{+}=\underset{n}{{\varprojlim}}\ \widetilde{\mathbb{B}}^{+}/(\ker\theta
)^{n}.$ The series $\log([\varepsilon])=-\sum\limits_{n=1}^{\infty}%
\frac{(1-[\varepsilon])^{n}}{n}$ converges to some element $t\in
\mathbb{B}_{dR}^{+}$ with the property that $gt=\chi(g)t$ for all $g\in G_{%
\mathbb{Q}
_{p}},$ where $\chi:G_{%
\mathbb{Q}
_{p}}$ $\rightarrow%
\mathbb{Z}
_{p}^{\times}$ is the cyclotomic character. We define $\mathbb{B}%
_{dR}=\mathbb{B}_{dR}^{+}[\frac{1}{t}].$ The ring $\mathbb{B}_{dR}$ is a field
equipped with a decreasing, exhaustive and separated filtration given by
Fil$^{j}\mathbb{B}_{dR}=t^{j}\mathbb{B}_{dR}^{+}$ for all integers $j.$ It
contains a subring $\mathbb{B}_{cris}$ endowed with the induced Galois action
and a Frobenius endomorphism $\varphi$ which extends $\varphi:\widetilde
{\mathbb{B}}^{+}\rightarrow\widetilde{\mathbb{B}}^{+},~$such that
$\varphi(t)=pt.$ It has the property that $\mathbb{B}_{cris}^{G_{K}}=K_{0}$
for any finite extension $K$ of $%
\mathbb{Q}
_{p},$ where $K_{0}$ is the maximum unramified extension of $%
\mathbb{Q}
_{p}$ inside $K.\ $\noindent Between $\mathbb{B}_{cris}$ and $\mathbb{B}_{dR}
$ sits (non canonically) a ring $\mathbb{B}_{st}=\mathbb{B}_{cris}[X],$ where
$X$ is a polynomial variable over $\mathbb{B}_{cris}.$ The ring $\mathbb{B}%
_{st}$ is equipped with a Frobenius which extends the Frobenius on
$\mathbb{B}_{cris}$ and is such that $\varphi(X)=pX.$ There is also a $\bar{%
\mathbb{Q}%
}_{p}$-linear monodromy operator $N=-\frac{d}{dX}$ which satisfies the
equation $N\varphi=p\varphi N.$ Let $\tilde{p}\in\tilde{E}^{+}$ be any element
with $\tilde{p}^{(0)}=p$ and let
\[
\log[\tilde{p}]=\log_{p}(p)-\sum\limits_{n=1}^{\infty}\frac{(1-[\tilde
{p}]/p)^{n-1}}{n}.
\]
There exist Galois equivariant, $\mathbb{B}_{cris}$-linear embeddings of
$\mathbb{B}_{st}$ in $\mathbb{B}_{dR}$ which map $X$ to $\log[\tilde{p}].$
They require a choice of $\log_{p}(p)$ and we always assume that $\log
_{p}(p)=0.$ The ring $\mathbb{B}_{st}$ is equipped with a Galois action which
extends the Galois action on $\mathbb{B}_{cris}.$ It has the properties
that$\ \mathbb{B}_{st}^{G_{K}}=K_{0}$ for any finite extension $K\ $of $%
\mathbb{Q}
_{p}\ $and the map $K\otimes_{K_{0}}\mathbb{B}_{st}^{G_{K}}\rightarrow
\mathbb{B}_{dR}$ is injective.

\subsection{Potentially semistable representations}

Let $K\ $be a finite extension of $%
\mathbb{Q}
_{p}\ $and $V$ a $%
\mathbb{Q}
_{p}$-linear representation of $G_{K}.$ The fact that $\mathbb{B}_{dR}^{G_{K}%
}=K$ is part of a technical condition called regularity which implies that the
$K$-vector space $D_{dR}(V)=(\mathbb{B}_{dR}\otimes_{%
\mathbb{Q}
_{p}}V)^{G_{K}}$ has dimension at most dim$_{%
\mathbb{Q}
_{p}}(V).$ The representation $V$ is called de Rham if equality holds. All
representations coming from geometry are de Rham. The $K$-space $D_{dR}(V)$ is
equipped with a natural decreasing, exhaustive and separated filtration given
by Fil$^{j}D_{dR}(V)=(t^{j}\mathbb{B}_{dR}^{+}\otimes_{%
\mathbb{Q}
_{p}}V)^{G_{K}}$ for any integer $j.$ An integer $j$ is called a Hodge-Tate
weight of a de Rham representation $V$ if Fil$^{-j}D_{dR}(V)\neq$
Fil$^{-j+1}D_{dR}(V),$ and is counted with multiplicity dim$_{K}\left(
\text{Fil}^{-j}D_{dR}(V)/\text{Fil}^{-j+1}D_{dR}(V)\right)  .$ There are $d=$
dim$_{%
\mathbb{Q}
_{p}}(V)$ Hodge-Tate weights for $V,$ counting multiplicities. A chosen
inclusion of $\mathbb{B}_{st}$ in $\mathbb{B}_{dR}$ defines (non canonically)
a filtration on $K\otimes_{K_{0}}D_{st}(V)=K\otimes_{K_{0}}(\mathbb{B}%
_{st}\otimes_{%
\mathbb{Q}
_{p}}V)^{G_{K}}$ which is preserved by the Galois action. By the construction
of the ring $\mathbb{B}_{st}\ $the\ inequality\ dim$_{K_{0}}D_{st}(V)\leq$
dim$_{%
\mathbb{Q}
_{p}}(V)$ always holds, and $V$ is called semistable when equality holds. It
is called potentially semistable if it becomes semistable when restricted to
$G_{L},$ for some finite extension $L$ of $K.\ $Crystalline representations
are semistable and semistable representations are de Rham, with the converse
inclusions being false. Potentially semistable representations are de Rham.
The converse is a difficult theorem of Berger (\cite{BE04b}), known as the $p
$-adic monodromy theorem.

Let $L$ be a finite Galois extension of $K$ and $E$ any finite extension of
$L.$ We write $D_{st}^{L}(V)$ instead of $D_{st}(V\mid_{G_{L}}).$ Assume that
$V$ is equipped with an $E$-linear structure which commutes with the $G_{K}%
$-action. The $L_{0}$-space$\ D_{st}^{L}(V)$ is additionally equipped with an
$L_{0}\otimes_{%
\mathbb{Q}
_{p}}E$-module structure, and $V$ is $L$-semistable if and only if $D_{st}%
^{L}(V)$ is free of rank dim$_{E}V.$ For the rest of the section we assume
that $V$ is $L$-semistable. The Frobenius endomorphism of $\mathbb{B}_{st}$
induces an automorphism $\varphi$ on $D_{st}^{L}(V)$ which is semilinear with
respect to the automorphism $\tau\otimes1_{E}$ of $L_{0}\otimes_{%
\mathbb{Q}
_{p}}E,$ where $\tau$ is the absolute Frobenius of $L_{0}.$ The monodromy
operator $N$ of $\mathbb{B}_{st}$ induces an $L_{0}\otimes_{%
\mathbb{Q}
_{p}}E$-linear nilpotent endomorphism $N$ on $D_{st}^{L}(V)$ such that
$N\varphi=p\varphi N.$ We equip $L\otimes_{L_{0}}D_{st}(V)$ with the
filtration induced by the injection $L\otimes_{L_{0}}D_{st}^{L}(V)\rightarrow
D_{dR}(V).$ It has the properties that Fil$^{j}\left(  L\otimes_{L_{0}}%
D_{st}^{L}(V)\right)  =0$ for $j\gg0$ and Fil$^{j}\left(  L\otimes_{L_{0}%
}D_{st}^{L}(V)\right)  =L\otimes_{L_{0}}D_{st}^{L}(V)$ for $j\ll0.$ The module
$D_{st}^{L}(V)$ is also equipped with an $L_{0}$-semilinear, $E$-linear action
of $G=$ Gal$(L/K)$ which commutes with $\varphi$ and $N$ and preserves the
filtration. The discussion above motivates the following.

\begin{definition}
\label{the filtered definition}A rank $n$ filtered $\left(  \varphi
,N,L/K,E\right)  $-module is a free module $D$ of rank $n$ over $L_{0}%
\otimes_{%
\mathbb{Q}
_{p}}E$ equipped with
\end{definition}

\begin{itemize}
\item an $L_{0}$-semilinear, $E$-linear automorphism $\varphi;$

\item an $L_{0}\otimes_{%
\mathbb{Q}
_{p}}E$-linear nilpotent endomorphism $N$ such that $N\varphi=p\varphi N;$

\item a decreasing filtration on $D_{L}=L\otimes_{L_{0}}D$ such that
Fil$^{j}D_{L}=0$ for $j\gg0$ and Fil$^{j}D_{L}=D_{L}$ for $j\ll0,$ and

\item an $L_{0}$-semilinear, $E$-linear action of $G=$ Gal$(L/K)$ which
commutes with $\varphi$ and $N$ and preserves the filtration of $D_{L}.$
\end{itemize}

\noindent A morphism of filtered $\left(  \varphi,N,L/K,E\right)  $-modules is
an $L_{0}\otimes_{%
\mathbb{Q}
_{p}}E$-linear map $h$ which commutes with $\varphi,$ $N$ and the
Gal$(L/K)$-action, and is such that the $L\otimes_{%
\mathbb{Q}
_{p}}E$-linear map $h_{L}=1_{L\otimes_{%
\mathbb{Q}
_{p}}E}\otimes h$ preserves the filtrations.\ A filtered $\left(
\varphi,N,L/K,E\right)  $-module is called weakly admissible if it is weakly
admissible as a filtered $(\varphi,N,E)$-module in the sense of
\cite[Cor.\ 3.1.2.1]{BM02}. \noindent The Galois action plays no role in weak
admissibility. \noindent We have the following fundamental theorem essentially
due to Colmez and Fontaine (cf. \cite[Cor. 3.1.1.3]{BM02}).

\begin{theorem}
Let $k\geq1$ be any integer. The category of $L$-semistable $E$-representa-

\noindent tions of $G_{K}$ with Hodge-Tate weights in the range
$\{0,1,...,k-1\}$ is equivalent to the category of weakly admissible filtered
$\left(  \varphi,N,L/K,E\right)  $-modules $D\ $such that \textnormal{Fil}%
$^{0}(D_{L})=D_{L}$ and \textnormal{Fil}$^{k}(D_{L})=0.$
\end{theorem}

\section{Rank two filtered $\left(  \varphi,N,L/K,E\right)  $-modules
\label{product ring}}

\noindent Throughout the paper $p$ will be a fixed prime number and $L/K$ any
finite Galois extension, with $K$ any finite extension of $%
\mathbb{Q}
_{p}.$ The coefficient field $E$ will be any finite, large enough extension of
$L.$ We denote by $m$ the degree of $L$ over $\mathcal{%
\mathbb{Q}
}_{p},$ by $f=[L_{0}:\mathcal{%
\mathbb{Q}
}_{p}]\ $the absolute inertia degree of $L,$ and by $e=[L:L_{0}]$ the absolute
ramification index of $L.\ $As in the introduction we denote by $L_{0}\ $the
maximal unramified extension of $%
\mathbb{Q}
_{p}$ inside $L.$ Let$\ \tau$ be the absolute Frobenius of $L_{0}.$ We fix an
embedding $\iota_{L_{0}}:L_{0}\hookrightarrow E$ and we let $\tau_{j}%
=\iota_{L_{0}}\circ\tau^{j}$ for all $j=0,1,...,f-1.$ We fix once and for all
the $f$-tuple of embeddings $\mathcal{S}_{L_{0}}:=(\tau_{0},\tau_{1}%
,...,\tau_{f-1}).$ The map%
\[
\xi_{L_{0}}:L_{0}\otimes_{%
\mathbb{Q}
_{p}}E\rightarrow\prod\limits_{\ \ \mathcal{S}_{L_{0}}}E\ :\ \xi_{L_{0}%
}(x\otimes y)=(\tau_{i}(x)y)_{\tau_{i}}%
\]
is a ring isomorphism (cf. \cite[Lemma 2.2]{SAV05}). Let $E^{\mid
\mathcal{S}_{L_{0}}\mid}:=\prod\limits_{\ \ \mathcal{S}_{L_{0}}}E\ $and
$\left(  E^{\times}\right)  ^{\mid\mathcal{S}_{L_{0}}\mid}:=\prod
\limits_{\ \ \mathcal{S}_{L_{0}}}E^{\times}.$ The ring automorphism
$\tau\otimes1_{E}:L_{0}\otimes_{%
\mathbb{Q}
_{p}}E\rightarrow L_{0}\otimes_{%
\mathbb{Q}
_{p}}E$ transforms via $\xi_{L_{0}}$ to the ring automorphism $\varphi
:E^{\mid\mathcal{S}_{L_{0}}\mid}\rightarrow E^{\mid\mathcal{S}_{L_{0}}\mid}$
with $\varphi(x_{0},x_{1},...,x_{f-1})$

\noindent$=(x_{1},...,x_{f-1},x_{0}).$ A filtered $\left(  \varphi
,N,L/K,E\right)  $-module may therefore be viewed as a module over
$E^{\mid\mathcal{S}_{L_{0}}\mid}.$ The automorphism $\varphi:D\rightarrow D$
is semilinear with respect to the automorphism $\varphi\ $of $E^{\mid
\mathcal{S}_{L_{0}}\mid}$ defined above, and the monodromy $N$ is
$E^{\mid\mathcal{S}_{L_{0}}\mid}$-linear. The Galois action of $G=$ Gal$(L/K)$
on $E^{\mid\mathcal{S}_{L_{0}}\mid}$ will be described in Section
\ref{action F_0}. We let $e_{\tau_{j}}:=(0,...,1_{\tau_{j}},...,0)\in
E^{\mid\mathcal{S}_{L_{0}}\mid}\ $for any $j\in\{0,1,...,f-1\},$ and set up
some more notation which will remain fixed throughout.

\begin{notation}
\noindent\label{first notation} For each $J\subset\{0,1,...,f-1\}$ we write
$f_{J}=\sum\limits_{i\in J}e_{_{\tau_{i}}}.$ \noindent If $\vec{x}\in
E^{\mid\mathcal{S}_{L_{0}}\mid},$ we define $Nm_{\varphi}(\vec{x}%
):=\prod\limits_{i=0}^{f-1}\varphi^{i}(\vec{x})\ $and$\ $$Tr_{\varphi}(\vec
{x}):=\sum\limits_{i=0}^{f-1}\varphi^{i}(\vec{x}).$ For any $\vec{x}\in
E^{\mid\mathcal{S}_{L_{0}}\mid}$ we denote by $x_{i}\ $the $i$-th component of
$\vec{x},$ and for any matrix $M\in M_{2}(E^{\mid\mathcal{S}_{L_{0}}\mid})$ we
write $Nm_{\varphi}(M)=M\varphi(M)\cdot\cdot\cdot\varphi^{f-1}(M),$ with
$\varphi$ acting on each entry of $M.$
\end{notation}

\subsection{Canonical forms for Frobenius and the monodromy
operator\label{canonical forms}}

We start by putting the matrix of Frobenius of a rank two $\varphi$-module in
a convenient form. The matrix of any (semi)linear operator $T$ on $D$ with
respect to an ordered basis $\underline{e}$ will be denoted by
$[T]_{\underline{e}}$ throughout. The following elementary lemma will be used frequently.

\begin{lemma}
\label{norm lemma}

\begin{enumerate}
\item[(1)] The operator $Nm_{\varphi}:\left(  E^{\times}\right)
^{\mid\mathcal{S}_{L_{0}}\mid}\rightarrow\left(  E^{\times}\right)
^{\mid\mathcal{S}_{L_{0}}\mid}$ is multiplicative;

\item[(2)] Let $\vec{\alpha},\vec{\beta}\in\left(  E^{\times}\right)
^{\mid\mathcal{S}_{L_{0}}\mid}.$ The equation $\vec{\alpha}\cdot\vec{\gamma
}=\vec{\beta}\cdot\varphi(\vec{\gamma})$ has nonzero solutions $\vec{\gamma
}\in E^{\mid\mathcal{S}_{L_{0}}\mid}\ $if and only if $Nm_{\varphi}%
(\vec{\alpha})=$$Nm_{\varphi}(\vec{\beta}).$ In this case, all the solutions
are \noindent$\vec{\gamma}=\gamma\left(  1,\frac{\alpha_{0}}{\beta_{0}}%
,\frac{\alpha_{0}\alpha_{1}}{\beta_{0}\beta_{1}},...,\frac{\alpha_{0}%
\alpha_{1}\cdots\alpha_{f-2}}{\beta_{0}\beta_{1}\cdots\beta_{f-2}}\right)  $
for any $\gamma\in E.$
\end{enumerate}
\end{lemma}

\begin{proof}
Straightforward.
\end{proof}

\noindent Let $D\ $be a rank two $\varphi$-module$\ $over $E^{\mid
\mathcal{S}_{L_{0}}\mid}$ and let $\underline{\eta}\ $and $\underline{e}$ be
ordered bases. Then $(\eta_{1},\eta_{2})=(e_{1},e_{2})M$ for some matrix $M\in
GL_{2}\left(  E^{\mid\mathcal{S}_{L_{0}}\mid}\right)  ,$ and we write
$M\ =[1]_{\underline{\eta}}^{\underline{e}}.$ It follows from Section
\ref{product ring} that $[\varphi]_{\underline{e}}=M[\varphi]_{\underline
{\eta}}\varphi(M)^{-1}.$ The main observation of this section is the following proposition.

\begin{proposition}
\label{new shape of phi}Let $D\ $be a rank two $\varphi$-module$\ $over
$E^{\mid\mathcal{S}_{L_{0}}\mid}.$ After enlarging $E$ if necessary, there
exists an ordered basis $\underline{\eta}$ of $D$ with respect to which the
matrix of Frobenius takes one of the following forms:

\begin{enumerate}
\item[(1)] $[\varphi]_{\underline{\eta}}=$ \textnormal{diag}$(\alpha\cdot\vec
{1},\delta\cdot\vec{1})$ for some $\alpha,\delta\in E^{\times}$ with
$\alpha^{f}\neq\delta^{f},$ or

\item[(2)] $[\varphi]_{\underline{\eta}}=$ \textnormal{diag}$(\alpha\cdot\vec
{1},\alpha\cdot\vec{1})$ for some $\alpha\in E^{\times},$ or

\item[(3)] $[\varphi]_{\underline{\eta}}=\left(
\begin{array}
[c]{cc}%
\alpha\cdot\vec{1} & \vec{0}\\
\vec{1} & \alpha\cdot\vec{1}%
\end{array}
\right)  $ for some $\alpha\in E^{\times}.$
\end{enumerate}
\end{proposition}

\noindent To prove Proposition $\ref{new shape of phi}$, we use the following lemma.

\begin{lemma}
\label{shape of phi}Let $D\ $be as in Proposition $\ref{new shape of phi}$.
After enlarging $E$ if necessary, the following hold:

\begin{enumerate}
\item[(1)] If $\varphi^{f}$ is not an $E^{\times}$-scalar times the identity
map, then there exists an ordered basis $\underline{\eta}$ of $D$ such that
$[\varphi]_{\underline{\eta}}=\left(
\begin{array}
[c]{ll}%
\vec{\varepsilon} & \vec{0}\\
\vec{\eta} & \vec{\theta}%
\end{array}
\right)  ,$ with the additional properties that:

\begin{enumerate}
\item[(a)] If $Nm_{\varphi}(\vec{\varepsilon})\neq$ $Nm_{\varphi}(\vec{\theta
}),$ then $\vec{\eta}=\vec{0}$ and

\item[(b)] If $Nm_{\varphi}(\vec{\varepsilon})=$ $Nm_{\varphi}(\vec{\theta}),$
then $\vec{\varepsilon}=\vec{\theta}$ and $\vec{\eta}_{\varphi}=\vec{1}, $
where $\vec{\eta}_{\varphi}$ is the $(2,1)$ entry of the matrix $Nm_{\varphi
}\left(  [\varphi]_{\underline{\eta}}\right)  .$
\end{enumerate}

\item[(2)] If $\varphi^{f}=\alpha\cdot\vec{id}\ $for some $\alpha\in
E^{\times},$ then there exists an ordered basis $\underline{\eta}$ of $D$ such
that $[\varphi]_{\underline{\eta}}=$ \textnormal{diag}$\left(  (\alpha
,1,...,1),(\alpha,1,...,1)\right)  .$
\end{enumerate}
\end{lemma}

\begin{proof}
${(1)\ }$Since $\varphi^{f}$ is an $E^{\mid\mathcal{S}_{L_{0}}\mid}$-linear
isomorphism, extending $E$ if necessary, there exists an ordered basis
$\underline{e}$ of $D$ such that $[\varphi^{f}]_{\underline{e}}=\left(
\begin{array}
[c]{ll}%
\vec{\alpha} & \vec{0}\\
\vec{\gamma} & \vec{\delta}%
\end{array}
\right)  .$ With the convention of Notation \ref{first notation} we have
$\alpha_{i}\delta_{i}\not =0$ for all $i\in I_{0}$ (because $\varphi$ is an
automorphism), and the basis can be chosen so that $\gamma_{i}=0$ whenever
$\alpha_{i}\not =\delta_{i}$ and $\gamma_{i}\in\{0,1\}$ whenever $\alpha
_{i}=\delta_{i}.$ We repeatedly act by $\varphi$ on the equation
$(\varphi(e_{1}),\varphi(e_{2}))=(e_{1},e_{2})[\varphi]_{\underline{e}}\, $
and get $\left(  \varphi^{f}(e_{1}),\varphi^{f}(e_{2})\right)  =(e_{1},e_{2}%
)$$Nm_{\varphi}([\varphi]_{\underline{e}}).$ Let $P=[\varphi]_{\underline{e}%
}=\left(  P_{0},P_{1},...,P_{f-1}\right)  $ and $Q=$ $Nm_{\varphi}(P)=\left(
Q_{0},Q_{1},...,Q_{f-1}\right)  .$ Since $Q=P\varphi(Q)P^{-1},$ we have
$Q_{i}=P_{i}Q_{i+1}P_{i}^{-1}$ and $\{\alpha_{i+1},\delta_{i+1}\}=\{\alpha
_{i},\delta_{i}\}~$for all $i.$ Since for all $i,$\ $\alpha_{i}\delta_{i}=\det
Q_{0}=d,$ we have $\{\alpha_{i+1},d\alpha_{i+1}^{-1}\}=\{\alpha_{i}%
,d\alpha_{i}^{-1}\}.$ Let $\alpha=d\alpha_{0}^{-1}.$ Then $\alpha_{i}%
\in\{\alpha,d\alpha^{-1}\}$ for all $i,$ and $Nm_{\varphi}(P)=\left(
\begin{array}
[c]{ll}%
(\alpha_{0},...,\alpha_{f-1}) & (0~,~.~.~.~,~0)\\
(\gamma_{0},...,\gamma_{f-1}) & (\delta_{0},...,\delta_{f-1})
\end{array}
\right)  $ with $\delta_{i}=d\alpha_{i}^{-1}.$ \noindent If $\alpha^{2}%
\not =d$ then, $\vec{\gamma}=\vec{0}$ and if $\alpha^{2}=d,$ then $\gamma
_{i}\in\{0,1\}\ $for all $i.$ We conjugate by the matrix $R=\left(
R_{0},R_{1},...,R_{f-1}\right)  ,$ where $R_{i}$\noindent$=\left(
\begin{array}
[c]{ll}%
1 & 0\\
0 & 1
\end{array}
\right)  $ or $\left(
\begin{array}
[c]{ll}%
0 & 1\\
1 & 0
\end{array}
\right)  $ depending on whether $\alpha_{i}=d\alpha^{-1}\ $or$\ \alpha$
respectively, and get $RQR^{-1}$\noindent$=\left(
\begin{array}
[c]{ll}%
d\alpha^{-1}\cdot\vec{1} & \ \vec{\gamma}\\
\ \ \ \ \ \ \vec{0} & \alpha\cdot\vec{1}%
\end{array}
\right)  .$ \noindent If $\alpha^{2}\not =d,$\ then $RQR^{-1}=$ diag$((d\alpha
^{-1},...,d\alpha^{-1}),(\alpha,\alpha,...,\alpha)).$ If $\alpha^{2}=d,$ then
$\mathnormal{Nm}{(P)}=\left(
\begin{array}
[c]{ll}%
\alpha\cdot\vec{1} & \ \ \ \vec{1}\\
\ \ \ \vec{0} & \alpha\cdot\vec{1}%
\end{array}
\right)  .\ $Indeed, since $P\varphi(Q)P^{-1}=Q,\ $if $\gamma_{j}=0$ for some
$j$ then $\gamma_{j+1}=0$ and $\varphi^{f}=\alpha\cdot i\vec{d}$ a
contradiction. Therefore $\vec{\gamma}=\vec{1}.$ We have proved that there
exists some ordered basis $\underline{\eta}$ of $D$ over $E^{\mid
\mathcal{S}_{L_{0}}\mid}$ such that $[\varphi^{f}]_{\underline{\eta}}=\left(
\begin{array}
[c]{ll}%
\alpha\cdot\vec{1} & \ \ \ \vec{0}\\
\gamma\cdot\vec{1} & \frac{d}{\alpha}\cdot\vec{1}%
\end{array}
\right)  $ for some $\alpha\in E^{\times}$ and some $\gamma\in E$ with
$\gamma=0$ if $\alpha^{2}\not =d$ and $\gamma=1$ if $\alpha^{2}=d.$ We compute
the matrix of $\varphi$ with respect to that basis $\underline{\eta}.$ The
relations ${Nm}_{\varphi}\left(  [\varphi]_{\underline{\eta}}\right)
=[\varphi^{f}]_{\underline{\eta}}$ and $[\varphi]_{\underline{\eta}}%
\varphi\left(  Nm_{\varphi}\left(  [\varphi]_{\underline{\eta}}\right)
\right)  =$ $Nm_{\varphi}\left(  [\varphi]_{\underline{\eta}}\right)
[\varphi]_{\underline{\eta}}$ and a direct computation imply that:
\noindent(1) If $\alpha^{2}\not =d,$ then the non diagonal entries of
$[\varphi]_{\underline{\eta}}$ are $\vec{0},$ \noindent and (2) If $\alpha
^{2}=d,$ then the $(1,2)\ $entry of $[\varphi]_{\underline{\eta}}\ $is
$\vec{0}$ and the diagonal entries are equal. This concludes the proof of part
(1). \noindent Part \noindent(2) follows immediately from the fact that the
matrix of $\varphi^{f}$ is basis-independent combined with the following claim.
\end{proof}

\begin{claim}
\textit{Let }$P\in GL_{2}(E^{\mid\mathcal{S}_{L_{0}}\mid})$\textit{\ be such
that Nm}$_{\varphi}(P)=$ \textnormal{diag\ }$(\alpha\cdot\vec{1},\alpha\cdot
\vec{1})$\textit{\ for some }$\alpha\in E^{\times}.$\textit{\ Then there
exists some matrix }$Q^{\ast}\in GL_{2}(E^{\mid\mathcal{S}_{L_{0}}\mid}%
)$\textit{\ such that }%
\[
Q^{\ast}P\varphi(Q^{\ast})^{-1}=\text{\textnormal{diag}}((\alpha
,1,..,1),\mathit{\noindent}(\alpha,1,..,1)).
\]

\end{claim}

\begin{proof}
As above we write $P=\left(  P_{0},P_{1},...,P_{f-1}\right)  .$ We easily see
that there exist matrices $Q_{i}\in GL_{2}(E)$ such that the matrix $Q=\left(
Q_{0},Q_{1},...,Q_{f-1}\right)  $ has the property \noindent$QP\varphi
(Q)^{-1}=\left(  T_{0},T_{1},...,T_{f-2},T_{f-1}\right)  $ for some triangular
matrices $T_{i}=\left(
\begin{array}
[c]{cc}%
\alpha_{i} & 0\\
\gamma_{i} & \delta_{i}%
\end{array}
\right)  $ for $i=0,1,...,f-2,$ and some matrix $T_{f-1}=\left(
\begin{array}
[c]{cc}%
\alpha_{f-1} & \beta_{f-1}\\
\gamma_{f-1} & \delta_{f-1}%
\end{array}
\right)  \in GL_{2}(E).\ $In the proof of this claim, the entries $\alpha
_{i},\beta_{i},\gamma_{i}$ and $\delta_{i}$ are having independent meaning and
should not be confused with those used before. \noindent The equation
$Nm_{\varphi}(QP\varphi(Q)^{-1})=$ diag$(\alpha\cdot\vec{1},\alpha\cdot\vec
{1})$ implies that $\prod\limits_{i=0}^{f-1}\alpha_{i}=\alpha$ and
$(\prod\limits_{i=0}^{f-2}\alpha_{i})\beta_{f-1}=0.$ Hence $\beta_{f-1}=0$ and
$QP\varphi(Q)^{-1}$\noindent\noindent$=\left(
\begin{array}
[c]{cc}%
\vec{\alpha} & \vec{0}\\
\vec{\gamma} & \vec{\delta}%
\end{array}
\right)  $ with \noindent$Nm_{\varphi}(\vec{\alpha})=$ $Nm_{\varphi}%
(\vec{\delta})=\alpha\cdot\vec{1}.$ \noindent Let $\vec{x}=(1,\alpha_{0}%
\alpha^{-1},\noindent\alpha_{0}\alpha_{1}\alpha^{-1},...,\alpha_{0}\alpha
_{1}\cdots\alpha_{f-2}\alpha^{-1}),\ $ \noindent$\vec{y}=(1,\delta_{0}%
\alpha^{-1},\delta_{0}\delta_{1}\alpha^{-1},...,\delta_{0}\delta_{1}%
\cdots\delta_{f-2}\alpha^{-1})$ \noindent and $R=$ diag$(\vec{x},\vec{y})\cdot
Q.$ A computation shows that
\[
RP\varphi(R)^{-1}=\left(
\begin{array}
[c]{cc}%
(\alpha,1,..,1) & \vec{0}\\
\vec{\zeta} & (\alpha,1,..,1)
\end{array}
\right)
\]
for some $\vec{\zeta}\in\left(  E\right)  ^{\mid\mathcal{S}_{L_{0}}\mid}.$
Since $Nm_{\varphi}(RP\varphi(R)^{-1})=$ diag$(\alpha\cdot\vec{1},\alpha
\cdot\vec{1})$ we have $\zeta_{0}+\alpha\sum\limits_{i=1}^{f-1}\zeta_{i}=0.$
\noindent Let $S=\left(
\begin{array}
[c]{cc}%
(1,1,...,1) & (0,0,...,0)\\
(z_{0},z_{1},...,z_{f-1}) & (1,1,...,1)
\end{array}
\right)  $, where $z_{0}=1,$\textit{\ }$z_{1}=1-\zeta_{1}-\zeta_{2}%
-\cdots-\zeta_{f-1},$ $z_{2}=1-\zeta_{2}-\cdots-\zeta_{f-1},...,$
$z_{f-2}=1-\zeta_{f-2}-\zeta_{f-1}$ and $z_{f-1}=1-\zeta_{f-1},$ and let
$Q^{\ast}=SR.$ The fact that $\zeta_{0}+\alpha\sum\limits_{i=1}^{f-1}\zeta
_{i}=0$ and a simple computation yield that $Q^{\ast}P\varphi(Q^{\ast})^{-1}=$
diag$\ ((\alpha,1,..,1),(\alpha,1,..,1)).$\textit{\ }
\end{proof}
\begin{proof}
[\textit{Proof of Proposition $\ref{new shape of phi}$}]Again, the notations
in the proof of this lemma are having independent meaning and should not be
confused with those of previous sections. Choose $\underline{\eta}\ $as in
Lemma \ref{shape of phi}. \noindent In case\textit{\ }(1)(a)$\ $so
that$\ [\varphi]_{\underline{\eta}}=$ diag$(\vec{\varepsilon},\vec{\theta})$
with $Nm_{\varphi}(\vec{\varepsilon})\neq$ $Nm_{\varphi}(\vec{\theta}),$ let
$\alpha_{1},\delta_{1}\in E^{\times}\ $be such that $Nm_{\varphi}%
(\vec{\varepsilon})=\alpha_{1}^{f}\cdot\vec{1}$ and $Nm_{\varphi}(\vec{\theta
})=\delta_{1}^{f}\cdot\vec{1}.$ By Lemma \ref{norm lemma} there exists a
matrix $M\in GL_{2}(E^{\mid\mathcal{S}_{L_{0}}\mid})$ such that $M\left(
[\varphi]_{\underline{\eta}}\right)  \varphi(M)^{-1}=$ diag$(\alpha_{1}%
\cdot\vec{1},\delta_{1}\cdot\vec{1}),$ and clearly $\alpha_{1}^{f}\neq
\delta_{1}^{f}.$ This gives the first possibility of the proposition.
\noindent In case (1)(b) of Lemma \ref{shape of phi}, let$\ \alpha_{1}~$an
$f$-th root of $\alpha.$ By Lemma \ref{norm lemma} there exists a matrix $M\in
GL_{2}(E^{\mid\mathcal{S}_{L_{0}}\mid})$ such that $M\left(  [\varphi
]_{\underline{\eta}}\right)  \varphi(M)^{-1}=\left(
\begin{array}
[c]{cc}%
\alpha_{1}\cdot\vec{1} & \vec{0}\\
\vec{\gamma} & \alpha_{1}\cdot\vec{1}%
\end{array}
\right)  .\ $ \noindent Since $[\varphi^{f}]_{\underline{\eta}}=\left(
\begin{array}
[c]{cc}%
\alpha_{1}^{f}\cdot\vec{1} & \vec{0}\\
\alpha_{1}^{f-1}Tr_{\varphi}(\vec{\gamma}) & \alpha_{1}^{f}\cdot\vec{1}%
\end{array}
\right)  $ and $[\varphi^{f}]_{\underline{e}}=\left(
\begin{array}
[c]{cc}%
\alpha\cdot\vec{1} & \vec{0}\\
\vec{1} & \alpha\cdot\vec{1}%
\end{array}
\right)  ,$ we have $Tr_{\varphi}(\vec{\gamma})\neq\vec{0}.\ $Let $M^{\ast
}=\left(
\begin{array}
[c]{cc}%
f\cdot\vec{1} & \vec{0}\\
\vec{z} & Tr_{\varphi}(\vec{\gamma})
\end{array}
\right)  ,\ $where%
\[
\vec{z}=\left(  0,1,...,f-1\right)  Tr_{\varphi}(\vec{\gamma})-f\left(
\gamma_{0},\gamma_{0}+\gamma_{1},...,\gamma_{0}+\gamma_{1}+\cdots\gamma
_{f-2}\right)  .
\]
\noindent Then $\left(
\begin{array}
[c]{cc}%
\alpha_{1}\cdot\vec{1} & \vec{0}\\
\vec{\gamma} & \alpha_{1}\cdot\vec{1}%
\end{array}
\right)  \varphi(M^{\ast})=M^{\ast}\left(
\begin{array}
[c]{cc}%
\alpha_{1}\cdot\vec{1} & \vec{0}\\
\vec{1} & \alpha_{1}\cdot\vec{1}%
\end{array}
\right)  .$ This gives the third possibility of the proposition. \noindent
Finally, in case (2)(b) of Lemma \ref{shape of phi}, let $\alpha_{1}\in
E^{\times}$ be an $f$-th root of $\alpha$ and proceed as in case (1).
\noindent This gives the second possibility of the proposition and concludes
the proof.
\end{proof}

\begin{definition}
\label{defin of f-diagonal...} A $\varphi$-module $D$ is called \textnormal{F}%
-semisimple, \textnormal{F}-scalar or non-\textnormal{F}-semisimple if and only if the
$E^{\mid\mathcal{S}_{L_{0}}\mid}$-linear map $\varphi^{f}$ has the
corresponding property.
\end{definition}

\noindent One easily sees that $D$ is \textnormal{F}-semisimple if and only if
there exists some ordered basis with respect to which the matrix of Frobenius
is as in cases (1) or (2) of Proposition $\ref{new shape of phi}$, with $D$
being non \textnormal{F}-scalar in case (1) and \textnormal{F}-scalar in case
(2).$\ $The $\varphi$-module $D$ is not \textnormal{F}-semisimple if and only if
there exists an ordered basis with respect to which the matrix of Frobenius is
as in case (3). A basis of $D$ in which Frobenius is normalized as in
Proposition $\ref{new shape of phi}$ will be called standard. Unless otherwise
stated, the matrix of any operator on $D$ will be considered with respect to a
fixed standard basis. In the next proposition we determine the matrix of the
monodromy operator with respect to a standard basis\ $\underline{\eta}.$

\begin{proposition}
\label{the monodromy} Let $D$ be a rank two $(\varphi,N,E)$-module.

\begin{enumerate}
\item If $D$ is \textnormal{F}-semisimple and $[\varphi]_{\underline{\eta}}=$
\textnormal{diag}$(\alpha\cdot\vec{1},\delta\cdot\vec{1}),$ then the monodromy
operator is as follows:

\begin{enumerate}
\item If $\alpha^{f}\neq p^{\pm f}\delta^{f},$ then $N=0;$

\item If $\alpha^{f}=p^{f}\delta^{f},$ then $[N]_{\underline{\eta}}=\left(
\begin{array}
[c]{cc}%
\vec{0} & \vec{0}\\
\vec{n} & \vec{0}%
\end{array}
\right)  ,$ where \noindent$\vec{n}=n(1,\zeta,\zeta^{2},...,\zeta^{f-1}),$
with $\zeta=\frac{\alpha}{p\delta}$ and $n\in E;$

\item If $\delta^{f}=p^{f}\alpha^{f},$ then $[N]_{\underline{\eta}}=\left(
\begin{array}
[c]{cc}%
\vec{0} & \vec{n}\\
\vec{0} & \vec{0}%
\end{array}
\right)  ,$ where \noindent$\vec{n}=n(1,\varepsilon,\varepsilon^{2}%
,...,\varepsilon^{f-1}),$ with $\varepsilon=\frac{\delta}{p\alpha}$ and $n\in
E.$
\end{enumerate}

\item If $D$ is non-\textnormal{F}-semisimple, then $N=0.$
\end{enumerate}
\end{proposition}

\begin{proof}
The condition $N\varphi=p\varphi N$ is equivalent to $[N]_{\underline{\eta}%
}[\varphi]_{\underline{\eta}}=p$ $[\varphi]_{\underline{\eta}}\varphi
([N]_{\underline{\eta}}).\ $The proposition follows by a short computation,
using Lemma \ref{norm lemma} and taking into account that $N$ is nilpotent.
\end{proof}

\begin{corollary}
\label{simplify N}Let $D$ be a rank two $(\varphi,N,E)$-module with nontrivial
monodromy. There exists an ordered basis $\underline{\eta}$ with respect to
which $[\varphi]_{\underline{\eta}}=$ \textnormal{diag}$(\alpha\cdot\vec{1}%
,\delta\cdot\vec{1})\ $for some $\alpha,\delta\in E^{\times}$ with
$\alpha=p\delta,$ and $[N]_{\underline{\eta}}=\left(
\begin{array}
[c]{cc}%
\vec{0} & \vec{0}\\
\vec{1} & \vec{0}%
\end{array}
\right)  .$
\end{corollary}

\begin{proof}
If $\alpha^{f}=p^{f}\delta^{f},$ change the basis to $\underline{\eta}%
^{\prime}$ with $\eta_{1}^{\prime}=\eta_{1}$ and $\eta_{2}^{\prime}=\vec
{n}\cdot\eta_{2}.$ If $\delta^{f}=p^{f}\alpha^{f},$ first swap the basis
elements, and then proceed as in the previous case.
\end{proof}

\noindent

When the monodromy operator is nontrivial our standard bases will always be as
in the corollary above.

\subsection{Galois descent data\label{gal descent}}

In this section we determine the action of the Galois group Gal$(L/K)$ on an
arbitrary rank two filtered $\left(  \varphi,N,L/K,E\right)  $-module $D.$

\subsubsection{The Galois action on $L\otimes_{ \mathbb{Q}_{p}}E$%
\label{action F}}

Since $E$ is assumed to be large enough, each embedding $\tau_{j}$ of $L_{0} $
into $E$ extends to an embedding of $L$ into $E$ in exactly $e=[L:L_{0}]$
different ways. For each $j\in\{0,1,...,f-1\},$ let $h_{ij}:L\rightarrow E$
with $i\in\{0,1,...,e-1\}$ be any numbering of the distinct extensions of
$\tau_{j}:L_{0}\rightarrow E$ to $L.$ Each index $s\in\{0,1,...,m-1\}$ can be
written uniquely in the form $s=fi+j$ with $i\in\{0,1,...,e-1\}$ and
$j\in\{0,1,...,f-1\}.$ For each $s=0,1,...,m-1,$ let $\sigma_{s}:=h_{ij}%
.\ $These are all the distinct embeddings of $L$ into $E$ and we fix the
$m$-tuple of embeddings $\mathcal{S}_{L}:=(\sigma_{0},\sigma_{1}%
,...,\sigma_{m-1})$ once and for all. Recall the notation $E^{\mid
\mathcal{S}_{L}\mid}:=\prod\limits_{\ \mathcal{S}_{L}}E.$ The map%
\[
\xi_{L}:L\otimes_{%
\mathbb{Q}
_{p}}E\rightarrow E^{\mid\mathcal{S}_{L}\mid}:x\otimes y\mapsto(\sigma
(x)y)_{\sigma}%
\]
is a ring isomorphism. A simple computation shows that $\xi_{L}(1\otimes
\alpha)=\xi_{L_{0}}(\alpha)^{\otimes e}\ $ for any $\alpha\in L_{0}\otimes_{%
\mathbb{Q}
_{p}}E,$ where $\xi_{L_{0}}$ is the isomorphism of Section $\ref{product ring}%
$. For each vector $\vec{a}\in E^{\mid\mathcal{S}_{L_{0}}\mid}$ we denote
$\vec{a}^{\otimes e}$ the vector of $E^{\mid\mathcal{S}_{L}\mid}$ gotten by e
copies of $\vec{a},$ removing the inner parentheses. For each $g\in
G=\text{Gal}(L/K)$ consider the permutation $\pi(g)$ on $\{0,1,...,m-1\}$
defined by $\sigma_{i}\cdot g=\sigma_{\pi(g)(i)}$ for any $g\in G$ and any
embedding $\sigma_{i}.$ The map $\rho:G\rightarrow S_{m}$ with $\rho
(g)=\pi(g)^{-1}$ is a group monomorphism. We define an $E$-linear $G$-action
on $E^{\mid\mathcal{S}_{L}\mid}$ by setting $g\xi_{L}(\alpha)=\xi_{L}%
(g\alpha)$ for all $g$ and $\alpha.$ If $x\otimes y\in L\otimes_{%
\mathbb{Q}
_{p}}E\ $and $g\in G,$ then $g\xi_{L}(x\otimes y)=(\sigma_{\pi(g)(i)}%
(x)y)_{\sigma_{i}},\ $therefore $g(\sigma_{0}(x)y,\sigma_{1}(x)y,...,\sigma
_{m-1}(x)y)=(\sigma_{\pi(g)(0)}(x)y,...,\sigma_{\pi(g)(m-1)}(x)y)$ for any
$x\otimes y\in L\otimes_{%
\mathbb{Q}
_{p}}E\ $(with indices viewed modulo $m$). From this we easily deduce that
\[
g(x_{0},x_{1},...,x_{m-1})=(x_{\pi(g)(0)},...,x_{\pi(g)(m-1)})
\]
\noindent for any $(x_{0},x_{1},...,x_{m-1})\in E^{\mid\mathcal{S}_{L}\mid}%
\ $and $g\in G.$

\subsubsection{The Galois action on $L_{0}\otimes_{%
\mathbb{Q}
_{p}}E$\label{action F_0}}

We use the isomorphism $\xi_{L_{0}}$ of Section $\ref{product ring}$ to define
an $E$-linear $G$-action on $E^{\mid\mathcal{S}_{L_{0}}\mid}$ by setting
$g\xi_{L_{0}}(x)=\xi_{L_{0}}(gx)$ for all $g\in G$ and $x\in L_{0}\otimes_{%
\mathbb{Q}
_{p}}E.$ For each $g\in G$ there exists a unique integer $n(g)\in
\{0,1,...,f-1\}$ such that $g\mid_{L_{0}}=\tau^{n(g)}.$ \noindent One easily
sees that $g\vec{\alpha}=(\alpha_{n(g)},\alpha_{n(g)+1},...,\alpha
_{n(g)+f-1})$ for all $g$ and $\vec{\alpha}=(\alpha_{0},\alpha_{1}%
,...,\alpha_{f-1}).$ We write $^{g}\vec{\alpha}$ instead of $g\vec{\alpha}$
and it is obvious that $Nm_{\varphi}(^{g}\vec{\alpha})=$$Nm_{\varphi}%
(\vec{\alpha}).$ \noindent Clearly $\xi_{L}(g(1\otimes\alpha))=\xi_{L_{0}%
}(g\alpha)^{\otimes e}$ for any $g\in G$ and $\alpha\in L_{0}\otimes_{%
\mathbb{Q}
_{p}}E,$ and this implies that $g(\vec{\alpha}^{\otimes e})=(g\vec{\alpha
})^{\otimes e}.$ In the next proposition we determine the matrix of the Galois
action with respect to a standard basis. Recall that when the monodromy is
nontrivial, standard bases are as in the comment succeeding Corollary
\ref{simplify N}.

\begin{proposition}
\label{g commutes with frob} Let $D$ be a rank two $\left(  \varphi
,N,L/K,E\right)  $-module and let $\underline{\eta}$ be a standard basis of
$D.$

\begin{enumerate}
\item If $D$ is \textnormal{F}-semisimple and non-scalar,

\begin{enumerate}
\item If the monodromy N is nontrivial, then there exists some $E^{\times}%
$-valued character $\chi\ $of $\mathnormal{G}$ such that$\ [g]_{\underline{\eta}%
}=$ \textnormal{diag}$(\chi(g)\cdot\vec{1},\chi(g)\cdot\vec{1})\ $for all
$g\in\mathnormal{G};$

\item If the monodromy N is trivial, then\ there exist some $E^{\times}%
$-valued characters $\chi,\psi\ $of $\mathnormal{G}$ such that $[g]_{\underline
{\eta}}=$ \textnormal{diag}$(\chi(g)\cdot\vec{1},\psi(g)\cdot\vec{1})$ for all
$g\in\mathnormal{G}.$
\end{enumerate}

\item If $D$ is \textnormal{F}-scalar, then there exists some group homomorphism

\noindent$\ \ \ \ \ \ \lambda:G\rightarrow GL_{2}(E)$ such that
$[g]_{\underline{\eta}}=\lambda(g)\cdot$\textnormal{diag}$(\vec{1},\vec{1})$ for
all $g\in G.$

\item If $D$ is not \textnormal{F}-semisimple, then\ there exist some $E^{\times}%
$-valued character $\chi\ $of \noindent\ \ \ \ \ \ $\ \ \ \ \ \ \mathnormal{G}$
such that \noindent$\lbrack g]_{\underline{\eta}}=$ \textnormal{diag}%
$(\chi(g)\cdot\vec{1},\chi(g)\cdot\vec{1})$ for all $g\in G.$
\end{enumerate}
\end{proposition}

\begin{proof}
For $G$ to act on $D$ we must have $[g_{1}g_{2}]_{\underline{\eta}}%
=[g_{1}]_{\underline{\eta}}\left(  ^{g_{1}}[g_{2}]_{\underline{\eta}}\right)
$ for any $g_{1},g_{2}\in G.$ We determine the shape of the matrices
$[g]_{\underline{\eta}}$ utilizing the fact that the Galois action commutes
with Frobenius and the monodromy operators.\ That happens if and only if
$[\varphi]_{\underline{\eta}}\varphi\left(  \lbrack g]_{\underline{\eta}%
}\right)  =[g]_{\underline{\eta}}(^{g}[\varphi]_{\underline{\eta}})$ and
\noindent$\lbrack N]_{\underline{\eta}}[g]_{\underline{\eta}}=[g]_{\underline
{\eta}}\left(  ^{g}[N]_{\underline{\eta}}\right)  $ for all $g\in G.$ The
proof of the proposition is a tedious calculation and we only give the details
in Case (3). For any $g,$ we write $[g]_{\underline{\eta}}=\left(
\begin{array}
[c]{cc}%
\vec{\alpha}(g) & \vec{\beta}(g)\\
\vec{\gamma}(g) & \vec{\delta}(g)
\end{array}
\right)  .$ In this case the monodromy operator is trivial. Let $[\varphi
]_{\underline{\eta}}=\left(
\begin{array}
[c]{cc}%
\alpha\cdot\vec{1} & \vec{0}\\
\vec{1} & \alpha\cdot\vec{1}%
\end{array}
\right)  $ for some $\alpha\in E^{\times}.$ The equation $[\varphi
]_{\underline{\eta}}\varphi\left(  \lbrack g]_{\underline{\eta}}\right)
=[g]_{\underline{\eta}}(^{g}[\varphi]_{\underline{\eta}})$ implies that for
all $g\in G,$ $[g]_{\underline{\eta}}=\left(
\begin{array}
[c]{cc}%
\alpha(g)\cdot\vec{1} & \vec{0}\\
\gamma(g)\cdot\vec{1} & \alpha(g)\cdot\vec{1}%
\end{array}
\right)  $ for some functions $\alpha,\gamma:G\rightarrow E.$ The equation
$[g_{1}g_{2}]_{\underline{\eta}}=[g_{1}]_{\underline{\eta}}\left(  ^{g_{1}%
}[g_{2}]_{\underline{\eta}}\right)  $ implies that $\alpha:G\rightarrow
E^{\times}$ is a character, and that $\gamma(g_{1}g_{2})=\alpha(g_{1}%
)\gamma(g_{2})+\alpha(g_{2})\gamma(g_{1})$ for all $g_{1}\ $and $g_{2}.$ By
induction, $\gamma(g^{n})=n\alpha(g^{n-1})\gamma(g)\ $for any $g\in G$ and any
non negative integer $n.$ Since $\gamma(1)=0$ and $\alpha(g)\neq0$ for all
$g,$ we have $\gamma(g)=0$ because $G$ is finite.
\end{proof}

\section{Galois-stable filtrations\label{stable filtrations}}

In this section we describe the shape of the filtrations of rank two filtered
modules and construct those which are stable under the Galois action. The
notion of a labeled Hodge-Tate weight will be important.

\subsection{Labeled Hodge-Tate weights}

If $D$ is a rank $n$ filtered $\left(  \varphi,N,L/K,E\right)  $-module,
$D_{L}=L\otimes_{L_{0}}D$ may be viewed as a module over $E^{\mid
\mathcal{S}_{L}\mid}$ via the ring isomorphism $\xi_{L}\ $of Section
$\ref{action F}$. For each embedding $\sigma$ of $L$ into $E,$ let $e_{\sigma
}:=(0,...,0,1_{\sigma},0,...,0)\in E^{\mid\mathcal{S}_{L}\mid}$ and
$D_{L,\sigma}:=e_{\sigma}D_{L}.$ We have the decomposition
\[
D_{L}=\bigoplus\limits_{\sigma\in S_{L}}D_{L,\sigma}.
\]
Since $D_{L}$ is free of rank $n$ over $L\otimes_{%
\mathbb{Q}
_{p}}E,$ the components $D_{L,\sigma}$ are equidimensional over $E,$ each of
dimension $n.$ We remark that the $E^{\mid\mathcal{S}_{L}\mid}$%
-modules$\ e_{\sigma}D_{L}$ are not necessarily free. We filter each component
$D_{L,\sigma}=e_{\sigma}D_{L}$ be setting Fil$^{j}D_{L,\sigma}:=e_{\sigma}
$Fil$^{j}D_{L}.$ An integer $j$ is called a labeled Hodge-Tate weight of
$D_{L}$ (or of $D$) with respect to the embedding $\sigma$ if and only if
Fil$^{-j}D_{L,\sigma}\neq$ Fil$^{-j+1}D_{L,\sigma}.$ It is counted with
multiplicity dim$_{E}\left(  \text{Fil}^{-j}D_{L,\sigma}/\text{Fil}%
^{-j+1}D_{L,\sigma}\right)  .$ Since the components $D_{L,\sigma}$ are
equidimensional over $E,$ there are $n$ labeled Hodge-Tate weights for each
embedding $\sigma,$ counting multiplicities. The labeled Hodge-Tate weights of
$D$ are by definition the $m$-tuple of multiset $(W_{i})_{\sigma_{i}},$ where
each such multiset $W_{i}$ contains $n$ integers, the opposites of the jumps
of the filtration of $D_{L,\sigma_{i}}.$ From now on we restrict attention to
rank two filtered modules with labeled Hodge-Tate weights $(\{0,-k_{i}%
\})_{\sigma_{i}},$ with $k_{i}$ non negative integers. When the labeled
Hodge-Tate weights are arbitrary, we can always shift them into this range,
after twisting by some appropriate rank one weakly admissible filtered
$\varphi$-module. Indeed, since Fil$^{j}\left(  D_{1}\otimes D_{2}\right)
=\sum\limits_{j_{1}+j_{2}=j}$Fil$^{j_{1}}D_{1}\otimes$Fil$^{j_{2}}D_{2}\ $ for
any filtered modules $D_{1}$ and $D_{2}$ and any integer $j,$ the claim
follows easily using the shape of the rank-one weakly admissible filtered
$\varphi$-modules given in the Appendix and the definition of a labeled
Hodge-Tate weight.

\begin{notation}
{\label{notation for filtrations}\noindent\ Let $k_{0},k_{1},...,k_{m-1}$ be
non negative integers which we call weights. Assume that after ordering them
and omitting possibly repeated weights we get $w_{0}<w_{1}<...<w_{t-1},$ where
$w_{0}$ is the smallest weight, \noindent$w_{1}$ the second smallest
weight,\ ..., $w_{t-1}$ is the largest weight and $1\leq t\leq m.$ }For
convenience we define $w_{-1}=0.${\ Let $I_{0}=\{0,1,...,m-1\},$ $I_{1}=\{i\in
I_{0}:k_{i}>w_{0}\},...,\ I_{t-1}=\{i\in I_{0}:k_{i}>w_{t-2}\}=\{i\in
I_{0}:k_{i}=w_{t-1}\},$ $I_{t}=\varnothing$ and $I_{0}^{+}=\{i\in I_{0}%
:k_{i}>0\}.$ Notice that $\sum\limits_{i=0}^{t-1}w_{i}(\mid I_{i}\mid-\mid
I_{i+1}\mid)=\sum\limits_{i=0}^{m-1}k_{i}.$ If $\vec{x}\in E^{\mid
\mathcal{S}_{L}\mid},$ we write $J_{\vec{x}}=\{i\in I_{0}:x_{i}\neq0\}.$ For
any $J\subset I_{0},$ we let $f_{J}:=\sum\limits_{i\in J}e_{\sigma_{i}}.$ If
}$A$ is a matrix with entries in{\ $E^{\mid\mathcal{S}_{L_{0}}\mid}$ }we write
$A^{\otimes e}${\ for the matrix with entries in $\prod
\limits_{\ \ \mathcal{S}_{L}}E$ obtained by replacing each entry $\vec{\alpha
}$ of }$A$ by $\vec{a}^{\otimes e},$ where $\vec{a}^{\otimes e}$ is as in
Section $\ref{action F}.$
\end{notation}

\subsection{The shape of the filtrations}

Let $D_{L}$ be a filtered $\varphi$-module with labeled Hodge-Tate weights
$(\{-k_{i},0\})_{\sigma_{i}}$ and let $\underline{\eta}=(\eta_{1},\eta_{2})$
be any ordered basis of $D\ $over $E^{\mid\mathcal{S}_{L_{0}}\mid}.$ By the
definition of a labeled Hodge-Tate weight we have
\[
\text{Fil}^{j}(D_{L,\sigma_{i}})=\left\{
\begin{array}
[c]{l}%
e_{\sigma_{i}}D_{L}\text{ ~\ \ if }j\leq0,\\
\ \ D_{L}^{i}\text{ ~~\ \ \ \ ~~~if }1\leq j\leq k_{i},\\
\ \ \ 0\text{ ~~~~~~~\ \ \ \ \ \ if }j\geq1+k_{i},
\end{array}
\right.
\]
where $D_{L}^{i}=(E^{\mid\mathcal{S}_{L}\mid})\left(  \vec{x}^{i}(1\otimes
\eta_{1})+\vec{y}^{i}(1\otimes\eta_{2})\right)  e_{_{\sigma_{i}}},$ for some
vectors $\vec{x}^{i}=(x_{0}^{i},x_{1}^{i},...,$

\noindent$x_{m-1}^{i})$ and \ $\vec{y}^{i}=(y_{0}^{i},y_{1}^{i},...,y_{m-1}%
^{i})$ $\in E^{\mid\mathcal{S}_{L}\mid},$ with the additional condition that
$(x_{i}^{i},y_{i}$$^{i})\not =(0,0)$ whenever $k_{i}>0.$ Since one may choose
the $x_{i}^{i}$ and $y_{i}^{i}$ arbitrarily when $k_{i}=0,$ we may assume that
$(x_{i}^{i},y_{i}^{i})\not =(0,0)$ for all $i\in I_{0}.$ From now on we always
make this assumption. Since Fil$^{j}(D_{L})=\bigoplus\limits_{i=0}%
^{m-1}e_{\sigma_{i}} $Fil$^{j}(D_{L}),$ we have Fil$^{j}D_{L}=D_{L}$ for
$j\leq0$ and Fil$^{j}D_{L}=0$ for $j\geq1+w_{t-1}.$ Let $1+w_{r-1}\leq j\leq
w_{r}$ for some $r\in\{0,1,...,t-1\}$ (recall that $w_{-1}=0$), then
Fil$^{j}D_{L}=\bigoplus\limits_{i\in I_{r}}D_{L}^{i}.$~If $\vec{x}=(x_{0}%
^{0},x_{1}^{1},...,x_{m-1}^{m-1})$ and $\vec{y}=(y_{0}^{0},y_{1}%
^{1},...,y_{m-1}^{m-1}),$ then $(x_{i}^{i},y_{i}^{i})\not =(0,0)$ for all
$i\in I_{0}$ and%

\[
\text{Fil}^{j}(D_{L})=\left\{
\begin{array}
[c]{l}%
\ \ D_{L}\ \ \ \text{if\ \ \ }j\,\leq0,\\
(E^{\mid\mathcal{S}_{L}\mid})f_{I_{0}}\left(  \vec{x}(1\otimes\eta_{1}%
)+\vec{y}(1\otimes\eta_{2})\right)  \text{\ \ \ \ if\ \ }1\leq j\leq w_{0},\\
(E^{\mid\mathcal{S}_{L}\mid})f_{I_{1}}\left(  \vec{x}(1\otimes\eta_{1}%
)+\vec{y}(1\otimes\eta_{2})\right)  \ \ \ \ \text{if\ \ }1+w_{0}\leq j\leq
w_{1},\\
\ \ \ \ \ \ \ \ \ \ \ \ \ \ \ \ \ \ \ \ \ \ \ \ \cdots\cdots\cdots\\
(E^{\mid\mathcal{S}_{L}\mid})f_{I_{t-1}}\left(  \vec{x}(1\otimes\eta_{1}%
)+\vec{y}(1\otimes\eta_{2})\right)  \ \text{ if\ \ }1+w_{t-2}\leq j\leq
w_{t-1},\\
\ \ \ 0\ \ \ \ \ \text{if\ \ }j\geq1+w_{t-1}.
\end{array}
\right.
\]

\begin{remark}
{\label{filtration}The filtration of $D_{L}$ can be put into the shape above
(for appropriate vectors $\vec{x}$ and $\vec{y}$) with respect to any ordered
basis of $D_{L}.$ Two filtrations of }$D_{L}${\ are called equivalent if one
is obtained from the other by replacing }$\vec{x}$ by $\vec{t}\cdot\vec{x}$
and $\vec{y}$ by $\vec{t}\cdot\vec{y},$ for some $\vec{t}\in\left(  E^{\times
}\right)  ^{\mid\mathcal{S}_{L}\mid}.$ Filtrations will be considered up to
equivalence and {one may assume that $\vec{y}=$ $f_{J_{\vec{y}}}.$ If
$\underline{\eta}=(\eta_{1},\eta_{2})$ is a standard basis of $D,\ $the
filtration of $D_{L}$ will be considered with respect to the basis
$1\otimes\underline{\eta}=(1\otimes\eta_{1},1\otimes\eta_{2}).$ We denote
$E^{\mid S_{L}\mid_{J}}:=(E^{\mid\mathcal{S}_{L}\mid})\cdot f_{J},$ for any
$J\subset I_{0}.$ }
\end{remark}

\subsection{Galois-stable filtrations in the non-\textnormal{F}-scalar case
\label{stable filtr}}

We now assume that $D$ is not \textnormal{F}-scalar and we construct the
filtrations of $D_{L}$ which are stable under the action of $G=\text{Gal}%
(L/K).$ We define a right action of $G$ on $I_{0}$ by letting$\ \ i\cdot
g:=\pi(g)(i),\ $where $\pi$ is as in Section \ref{action F}. Each orbit$\ $has
cardinality equal to $\#G,$ hence there are $\nu:=\left[  K:%
\mathbb{Q}
_{p}\right]  \ $orbits which we denote by $\mathcal{O}_{1},\mathcal{O}%
_{2},...,\mathcal{O}_{\nu}.$ Since the homomorphism $\rho$ of Section
\ref{action F} is injective, the $G$-action on $I_{0}\ $is free. \noindent Let
$[g]_{\underline{\eta}}=(\chi(g)\cdot\vec{1},\psi(g)\cdot\vec{1})\ $ \noindent
with the characters $\chi\ $and $\psi$ as in Proposition
\ref{g commutes with frob}, and let the filtration of $D_{L}$ be%

\begin{equation}
\text{Fil}^{j}(D_{L})=\left\{
\begin{array}
[c]{l}%
\ \ \ \ \ \ \ \ \ \ \ \ D_{L}\ \ \ \ \ \ \ \ \ \ \text{if\ \ \ }j\,\leq0,\\
\left(  E^{\mid S_{L}\mid_{I_{r}}}\right)  \left(  \vec{x}(1\otimes\eta
_{1})+\vec{y}(1\otimes\eta_{2})\right)  \text{\ if}\\
1+w_{r-1}\leq j\leq w_{r},\ \text{for}\ r=0,...,t-1,\\
\ \ \ \ \ \ \ \ \ \ \ \ \ 0\ \ \text{\ \ \ \ \ \ \ \ \ \ if \ \ }%
j\geq1+w_{t-1},
\end{array}
\right.  \label{filtrations}%
\end{equation}
for some vectors $\vec{x},\vec{y}\in E^{\mid\mathcal{S}_{L}\mid}$ with
$(x_{i},y_{i})\neq$ $(0,0)$ for all $i\in I_{0}.$ We must have that
$g($Fil$^{j}D_{L})\subset$ Fil$^{j}D_{L}$ for any $g\in G$ and $j\in%
\mathbb{Z}
.$ For any $r\in\{0,1,...,t-1\}$ there must exist some vector $\vec{t}=\vec
{t}(r,g)\in E^{\mid\mathcal{S}_{L}\mid}$ such that the following equations
hold:$\ $%
\begin{equation}
\chi(g)(^{g}f_{I_{r}\cap J_{\vec{x}}})\cdot(^{g}\vec{x})=\vec{t}\cdot
f_{I_{r}\cap J_{\vec{x}}}\cdot\vec{x}\ \text{and\ }\psi(g)(^{g}f_{I_{r}\cap
J_{\vec{y}}})\cdot^{g}\vec{y}=\vec{t}\cdot f_{I_{r}\cap J_{\vec{y}}}\cdot
\vec{y}. \label{g-stable equations}%
\end{equation}

\begin{notation}
If $g\in G$ and $J\subset I_{0}$ we denote by $^{g}J$ the set $\{j\cdot
g,\ j\in J\}.$
\end{notation}

\noindent\noindent For any $J,J_{1},J_{2}\subset I_{0},\ $any $g\in G\ $and
any $\vec{x}\in E^{\mid S_{L}\mid}\ $the following equations are trivial to
check:%
\begin{equation}%
\begin{array}
[c]{c}%
f_{J_{1}}\cdot f_{J_{2}}=f_{J_{1}\cap J_{2}},\ ^{g}(f_{I})=f_{(^{g}I)}%
,\ (^{g}f_{J_{1}})\cdot f_{J_{2}}=f_{(^{g}J_{1})\cap J_{2}},^{g}J_{\vec{x}%
}=J_{^{g}\vec{x}}\\
\text{and\ }^{g}(J_{1}\cap J_{2})=(^{g}J_{1})\cap(^{g}J_{2}).
\end{array}
\label{J^g}%
\end{equation}
Since $\chi(g)\neq0$ for all $g,\ $the equation $\chi(g)(^{g}f_{I_{r}\cap
J_{\vec{x}}})\cdot(^{g}\vec{x})=\vec{t}\cdot f_{I_{r}\cap J_{\vec{x}}}%
\cdot\vec{x}$ implies $^{g}(I_{r}\cap J_{\vec{x}})\cap J_{^{g}\vec{x}}\subset
I_{r}\cap J_{\vec{x}}.$ This is equivalent to $^{g}(I_{r}\cap J_{\vec{x}%
})\subset I_{r}\cap J_{\vec{x}}$ and therefore to $^{g}(I_{r}\cap J_{\vec{x}%
})=I_{r}\cap J_{\vec{x}}\ $for all $g\in G.$ Similarly, $^{g}(I_{r}\cap
J_{\vec{y}})=I_{r}\cap J_{\vec{y}}$ for all $g\in G.\ $The latter (for
$r=0\ $combined with Formulae (\ref{J^g})) imply that the sets $J_{\vec{x}}$
and $J_{\vec{y}}$ are $G$-stable and therefore unions of $G$-orbits of
$I_{0}.$ Since $J_{\vec{x}}\cup J_{\vec{y}}=I_{0},$ each set $I_{r}\ $is
$G$-stable and therefore a union of $G$-orbits as well. \noindent For a fixed
$g,$ equations (\ref{g-stable equations}) hold for any $r=0,1,...,t-1$ if and
only if they hold for $r=0,$ they are therefore equivalent to the existence of
some vector $\vec{t}=\vec{t}(g)\in E^{\mid\mathcal{S}_{L}\mid}$ such that
\[
\left(  x_{\pi(g)(i_{j})},y_{\pi(g)(i_{j})}\right)  =\left(  t(g)_{i_{j}}\cdot
x_{i_{j}},t(g)_{i_{j}}\cdot y_{i_{j}}\right)  \cdot\text{\textnormal{diag}}\left(
\chi(g)^{-1},\psi(g)^{-1}\right)  \ \text{for\ all}\ g\in G.
\]
Since $J_{\vec{x}}\cup J_{\vec{y}}=I_{0}$ all the coordinates of $\vec{t}(g)$
are non zero and by Remark \ref{filtration} we may assume that\ $\vec
{t}(g)=\vec{1}$ for all $g\in G.$ Let $i_{j}$ be any index in the orbit
$\mathcal{O}_{j},$ with $1\leq j\leq\nu,$ and let $(x_{i_{j}},y_{i_{j}})\in
E\times E$ with $(x_{i_{j}},y_{i_{j}})\neq(0,0).$ Since $G$ acts freely on
$I_{0},$ for each index $\ell\in I_{0}$ there exist unique $j\in
\{1,2,...,\nu\}$ and $g\in G$ such that $\ell=i_{j}\cdot g.$ Let $\vec{x}%
,\vec{y}\in E^{\mid\mathcal{S}_{L}\mid}$ be the vectors with coordinates
$\left(  x_{\ell},y_{\ell}\right)  :=\left(  x_{i_{j}},y_{i_{j}}\right)
\cdot$diag$\left(  \chi(g)^{-1},\psi(g)^{-1}\right)  $ for all $g\in G.$
Clearly
\[
\vec{x}=\sum\limits_{j=1}^{\nu}\left\{  \sum\limits_{g\in G}x_{i_{j}}\cdot
\chi(g^{-1})\cdot e_{\pi(g)(i_{j})}\right\}  \ \text{and}\ \vec{y}%
=\sum\limits_{j=1}^{\nu}\left\{  \sum\limits_{g\in G}y_{i_{j}}\cdot\psi
(g^{-1})\cdot e_{\pi(g)(i_{j})}\right\}  .
\]
By the discussion above we have the following proposition.

\begin{proposition}
The filtration in $\left(  \ref{filtrations}\right)  $ with vectors $\vec{x}$
and $\vec{y}$ as above is $G$-stable if and only if the sets $I_{r}\ $are
unions of $G$-orbits of $I_{0}$ for all $1\leq r\leq t-1.$ Conversely, any $G
$-stable filtration of $D_{L}$ is equivalent to a filtration of this form.
\end{proposition}

\begin{example}
{Let $K=%
\mathbb{Q}
_{p}$ and let $L\ $be any finite Galois extension of} $%
\mathbb{Q}
_{p}.$ The action of $G$ on $I_{0}$ is free and transitive. Since the sets
$I_{r}$ are unions of $G$-orbits, $I_{r}=\varnothing$ for all $r\geq1$ and all
the labeled Hodge-Tate weights are equal to some non negative integer $k. $
Since the sets $J_{\vec{x}}$ and $J_{\vec{y}}$ are unions of $G$-orbits, the
only possibilities are $(J_{\vec{x}},J_{\vec{y}})=(\varnothing,I_{0}),$
$(I_{0},\varnothing),$ $(I_{0},I_{0}).$ The only $G$-stable filtrations (up to
equivalence) are
\[
\text{\textnormal{Fil}}^{j}(D_{L})=\left\{
\begin{array}
[c]{l}%
\ \ D_{L}\ \ \ \ \ \text{if }j\leq0,\\
(E^{\mid\mathcal{S}_{L}\mid})\left(  \vec{x}(1\otimes\eta_{1})\ +\vec
{y}(1\otimes\eta_{2})\right)  \ \text{if }1\leq j\leq k,\\
\ \ 0\ \ \ \ \ \ \ \ \text{if }j\geq1+k,
\end{array}
\right.
\]
with $(\vec{x},\vec{y})=(\vec{0},\vec{1})$ if $(J_{\vec{x}},J_{\vec{y}%
})=(\varnothing,I_{0}),$ $(\vec{x},\vec{y})=(\vec{1},\vec{0})$ if $(J_{\vec
{x}},J_{\vec{y}})=(I_{0},\varnothing)\ $and
\[
\noindent(\vec{x},\vec{y})=\left(  x_{0}\left(  1,\frac{\psi(\sigma)}%
{\chi(\sigma)},\left(  \frac{\psi(\sigma)}{\chi(\sigma)}\right)
^{2},...,\left(  \frac{\psi(\sigma)}{\chi(\sigma)}\right)  ^{m-1}\right)
,\ \vec{1}\right)
\]
for any $x_{0}\in E^{\times},$ if $(J_{\vec{x}},J_{\vec{y}})=(I_{0},I_{0}).$
\end{example}

\subsection{Galois-stable filtrations in the \textnormal{F}-scalar case}

Let $\lambda$ be the homomorphism of Proposition \ref{g commutes with frob}
and let $\lambda(g)=\left(
\begin{array}
[c]{cc}%
\alpha(g) & \beta(g)\\
\gamma(g) & \delta(g)
\end{array}
\right)  .$ The Galois action preserves the filtration if and only if for any
$g\in G$ and any $0\leq r\leq t-1,$ there exists some vector $\vec{t}=\vec
{t}(g,r)\in E^{\mid\mathcal{S}_{L}\mid}$ such that%
\[%
\begin{array}
[c]{c}%
^{g}f_{I_{r}}\left\{  \alpha(g)\cdot\left(  ^{g}\vec{x}\right)  +\beta
(g)\cdot\left(  ^{g}\vec{y}\right)  \right\}  =\vec{t}\cdot\vec{x}\cdot
f_{I_{r}},\\
^{g}f_{I_{r}}\left\{  \gamma(g)\cdot\left(  ^{g}\vec{x}\right)  +\delta
(g)\cdot\left(  ^{g}\vec{y}\right)  \right\}  =\vec{t}\cdot\vec{y}\cdot
f_{I_{r}}.
\end{array}
\]
Suppose that there exists some $i\in\ ^{g}I_{r}$ with $i\not \in I_{r}.$ Then
$(x_{\pi(g)(i)},y_{\pi(g)(i)})\cdot\lambda(g)=(0,0),$ and since \textnormal{det}%
$\lambda(g)\neq0$ we have $(x_{\pi(g)(i)},y_{\pi(g)(i)})=(0,0)$ a
contradiction. Therefore $^{g}I_{r}=I_{r}$ for all $g.$ Then $g\left(
\text{Fil}^{i}D_{L}\right)  \subset$ Fil$^{j}D_{L}$ if and only if there
exists some vector $\vec{t}=\vec{t}(g,0)\in E^{\mid\mathcal{S}_{L}\mid}$ such
that $\left(  ^{g}\vec{x},^{g}\vec{y}\right)  =\left(  \vec{t}\cdot\vec
{x},\vec{t}\cdot\vec{y}\right)  \left(  \lambda(g^{-1})\cdot\text{diag}\left(
\vec{1},\vec{1}\right)  \right)  .$ This is equivalent to $\left(
x_{\pi(g)(i_{j})},y_{\pi(g)(i_{j})}\right)  =\left(  t(g)_{i_{j}}\cdot
x_{i_{j}},t(g)_{i_{j}}\cdot y_{i_{j}}\right)  \cdot\lambda(g^{-1})$ for all
$g\in G.$ Arguing as in Section \ref{stable filtr} one sees that $\vec
{t}(g,0)\in\left(  E^{\times}\right)  ^{\mid\mathcal{S}_{L}\mid}$ for all $g.$
By Remark \ref{filtration} we may assume that $\vec{t}(g)=\vec{1}$ for all
$g\in G.$ Let $i_{j}$ be any index in the orbit $\mathcal{O}_{j},$ with $1\leq
j\leq\nu,$ and let $(x_{i_{j}},y_{i_{j}})\in E\times E$ with $(x_{i_{j}%
},y_{i_{j}})\neq(0,0).$ Since $G$ acts freely on $I_{0},$ for each index
$\ell\in I_{0}$ there exist unique $j\in\{1,2,...,\nu\}$ and $g\in G $ such
that $\ell=i_{j}\cdot g.$ Let $\vec{x},\vec{y}\in E^{\mid\mathcal{S}_{L}\mid}$
be the vectors with coordinates $\left(  x_{\ell},y_{\ell}\right)  :=\left(
x_{i_{j}},y_{i_{j}}\right)  \cdot\lambda(g^{-1})$ for all $g\in G.$ Clearly
\[
\vec{x}=\sum\limits_{j=1}^{\nu}\left\{  \sum\limits_{g\in G}x_{\pi(g)(i_{j}%
)}\cdot e_{\pi(g)(i_{j})}\right\}  \ \text{and}\ \vec{y}=\sum\limits_{j=1}%
^{\nu}\left\{  \sum\limits_{g\in G}y_{\pi(g)(i_{j})}\cdot e_{\pi(g)(i_{j}%
)}\right\}  .
\]
By the discussion above we have the following proposition.

\begin{proposition}
The filtration in $\left(  \ref{filtrations}\right)  $ with vectors $\vec{x}$
and $\vec{y}$ as above is $G$-stable if and only if the sets $I_{r}$ are
unions of $G$-orbits of $I_{0}$ for all $1\leq r\leq t-1.$ Conversely, any $G
$-stable filtration of $D_{L}$ is equivalent to a filtration of this form.
\end{proposition}

\section{Hodge and Newton invariants\label{the fixed submodules}}

In this section we compute Hodge and Newton invariants of rank two filtered
$\varphi$-modules $\left(  D,\varphi\right)  .$ We thank the referee for
pointing out a mistake in the computation of Newton invariants. The same
mistake had been pointed out by David Savitt to whom we extend our thanks.

\noindent Let $v_{p}$ be the valuation of $\bar{%
\mathbb{Q}%
}_{p}$ normalized so that $v_{p}(p)=1$ and let val$_{L}(x)=ev_{p}(x)$ for any
$x\in L.$ Following \cite[\S 3]{BS06}, we define
\begin{equation}
t_{N}(D):=\dfrac{1}{[L:%
\mathbb{Q}
_{p}]}\text{val}_{L}\left(  \text{det}_{L_{0}}\varphi^{f}\right)
\label{t newton}%
\end{equation}
and
\begin{equation}
t_{H}(D_{L}):=\sum\limits_{\sigma\in S_{L}}\sum\limits_{j\in%
\mathbb{Z}
}\left(  \text{Fil}^{j}D_{L,\sigma}/\text{Fil}^{j+1}D_{L,\sigma}\right)  .
\label{t hodge}%
\end{equation}
Recall that the map $\varphi^{f}$ is $L_{0}\otimes_{%
\mathbb{Q}
_{p}}E$-linear. The filtered $\varphi$-module $(D,\varphi)$ is weakly
admissible if $t_{H}(D_{L})=t_{N}(D)$ and $t_{H}(D_{L}^{\prime})\leq
t_{N}(D^{\prime})$ for any $\varphi$-stable $L_{0}$-subspace $D^{\prime
}\subseteq D,\ $where $D_{L}^{\prime}=L\otimes_{L_{0}}D^{\prime},\ $and
$D_{L}^{\prime}$ is equipped with the induced filtration. By \cite[Prop.
3.1.1.5]{BM02} (with trivial modifications adopted to our definitions of the
Hodge and Newton invariants), one may only check the inequalities above for
$\varphi$-stable $L_{0}\otimes_{%
\mathbb{Q}
_{p}}E$-submodules $D^{\prime}$ of $D.$ We first determine the $L_{0}\otimes_{%
\mathbb{Q}
_{p}}E$-submodules of $D$ which are stable under Frobenius and the monodromy.

\begin{proposition}
\label{D(theta ) proposition}Let $\underline{\eta}=(\eta_{1},\eta_{2})\ \ $ be
an ordered basis with respect to which the matrix of Frobenius has the form
$[\varphi]_{\underline{\eta}}=\left(
\begin{array}
[c]{ll}%
\vec{\alpha} & \vec{0}\\
\vec{\gamma} & \vec{\delta}%
\end{array}
\right)  .$ All the $\varphi$-stable $L_{0}\otimes_{%
\mathbb{Q}
_{p}}E$-submodules of $D$ are $0,~D,~D_{2}=(E^{\mid\mathcal{S}_{L_{0}}\mid
})\eta_{2},$ or of the form $D_{\vec{\theta}}=(E^{\mid\mathcal{S}_{L_{0}}\mid
})(\eta_{1}+\vec{\theta}\eta_{2})$ for some vector $\vec{\theta}\in
E^{\mid\mathcal{S}_{L_{0}}\mid}.$
\end{proposition}

\begin{proof}
Let $M$ be a $\varphi$-stable submodule of $D.\ $\noindent Case (1). If
$M\cap(E^{\mid\mathcal{S}_{L_{0}}\mid})\eta_{2}\not =0.$ Let $\vec{x}\eta
_{2}\in M$ with $\vec{x}\not =\vec{0}.$ Then $\sum\limits_{i\in J_{\vec{x}}%
}e_{\tau_{i}}\eta_{2}\in M$, and after multiplying by $e_{\tau_{i}}$ for some
$i\in J_{\vec{x}}$ we get $e_{\tau_{i}}\eta_{2}\in M$ for some (in fact all)
$i\in J_{\vec{x}}.$ We repeatedly act by $\varphi$ and see that $e_{\tau_{i}%
}\eta_{2}\in M$ for all $i,$ which implies that $\eta_{2}\in M.$ If $\vec
{x}\eta_{1}+\vec{y}\eta_{2}\in M$ for some $\vec{x}\not =\vec{0},$ then
$\vec{x}\eta_{1}\in M. $ Arguing as before, given that $\eta_{2}\in M,$ we see
that $\eta_{1}\in M$ therefore $M=D.$ Hence in this case $M=(E^{\mid
\mathcal{S}_{L_{0}}\mid})\eta_{2}\ $or $M=D.$ Case\ (2). If $M\cap
(E^{\mid\mathcal{S}_{L_{0}}\mid})\eta_{2}=0.$ Assume that $M\neq0$ and let
$\vec{x}\eta_{1}+\vec{y}\eta_{2}\in M$ with $\vec{x}\not =\vec{0}.$ Then
$(\sum\limits_{i\in J_{\vec{x}}}e_{\tau_{i}})\eta_{1}+\vec{y}_{1}\eta_{2}\in
M$ for some $\vec{y}_{1}\in E^{\mid\mathcal{S}_{L_{0}}\mid}$and $e_{\tau_{i}%
}\eta_{1}+\vec{y}_{2}\eta_{2}\in M$ for some index $i\in J_{\vec{x}}$ and some
vector $\vec{y}_{2}.$ We repeatedly act by $\varphi$ and use the fact that $M$
is $\varphi$-stable to get that $\eta_{1}+\vec{\theta}\eta_{2}\in M$ for some
vector $\vec{\theta}.$ We will show that $M=(E^{\mid\mathcal{S}_{L_{0}}\mid
})(\eta_{1}+\vec{\theta}\eta_{2}).$ Every nonzero element of $M$ has the form
$\vec{\alpha}\eta_{1}+\vec{\beta}\eta_{2}$ for some vectors $\vec{\alpha
}\not =\vec{0}\ $and $\vec{\beta}.$ Since $\vec{\alpha}\eta_{1}+\vec{\alpha
}\cdot\vec{\theta}\eta_{2}$ $\in M,$ we see that $(\vec{\alpha}\cdot
\vec{\theta}-\vec{\beta})\eta_{2}\in M$ which implies that $\vec{\alpha}%
\cdot\vec{\theta}=\vec{\beta}.$ Then $\vec{\alpha}\eta_{1}+\vec{\beta}\eta
_{2}=\vec{\alpha}\eta_{1}+\vec{\alpha}\cdot\vec{\theta}\eta_{2}=\vec{\alpha
}(\eta_{1}+\vec{\theta}\eta_{2}).$
\end{proof}

\noindent We now determine the vectors $\vec{\theta}$ for which $D_{\vec
{\theta}}=(E^{\mid\mathcal{S}_{L_{0}}\mid})(\eta_{1}+\vec{\theta}\eta_{2}) $
is $\varphi$-stable. We have the following cases.

Case (1). If $D$ is \textnormal{F}-semisimple and non-scalar. In this case
$D_{\vec{\theta}}$ is $\varphi$-stable if and only if there exists $\vec{t}\in
E^{\mid\mathcal{S}_{L_{0}}\mid}$ such that $\varphi(\eta_{1}+\vec{\theta}%
\eta_{2})=\vec{t}(\eta_{1}+\vec{\theta}$ $\eta_{2}).$ We repeatedly act by
$\varphi$ and get $\varphi^{f}(\eta_{1})+\vec{\theta}\varphi^{f}(\eta_{2}%
)=$$Nm_{\varphi}(\vec{t})(\eta_{1}+\vec{\theta}\eta_{2}).$ This implies
$Nm_{\varphi}(\alpha\cdot\vec{1})=$$Nm_{\varphi}(\vec{t})$ and$\ \vec
{0}=(\alpha^{f}-\delta^{f})\cdot\vec{\theta}.$ Since $\alpha^{f}\not =%
\delta^{f},\ $the only nontrivial $\varphi$-stable submodules of $D$ are
$D_{1}=(E^{\mid\mathcal{S}_{L_{0}}\mid})\eta_{1}$ and $D_{2}=(E^{\mid
\mathcal{S}_{L_{0}}\mid})\eta_{2}.$

Case (2). If $D$ is \textnormal{F}-scalar we easily see that $D_{\vec{\theta}}$ is
$\varphi$-stable if and only if $\vec{\theta}=\theta\cdot\vec{1}$ for some
$\theta\in E^{\times}.$

Case (3). If $D$ is not \textnormal{F}-semisimple$\ D_{\vec{\theta}}$ is never
$\varphi$-stable.

\noindent Note that the submodules $D_{1},\ D_{2}$ and $D_{\theta}$ are
pairwise complementary in $D,$ and so are $D_{\theta_{1}}$ and $D_{\theta_{2}%
}$ whenever $\theta_{1}\neq\theta_{2}.$ \noindent Combining the results of the
previous paragraph with those of Proposition \ref{the monodromy}, we get the
following proposition.

\begin{proposition}
\label{something}Let $\underline{\eta}$ be a standard basis of a$\ (\varphi
,N)$-module$\ D.$ The submodules of $D$ fixed by Frobenius and the\ monodromy are\

\begin{enumerate}
\item $0,$ $D,$ $D_{1}=(E^{\mid\mathcal{S}_{L_{0}}\mid})\eta_{1}$ and
$D_{2}=(E^{\mid\mathcal{S}_{L_{0}}\mid})\eta_{2}\ $if $D$ is F-semisimple,
non-\textnormal{F}-scalar;

\item $0,$ $D,$ $D_{1},$ $D_{2}$ and $D_{\theta}=(E^{\mid\mathcal{S}_{L_{0}%
}\mid})(\eta_{1}+\theta\cdot\vec{1}\cdot\eta_{2}),$ for any $\theta\in
E^{\times}$ if $D$ is \textnormal{F}-scalar;

\item $0,$ $D$ and $D_{2}$ if $D$ is \textnormal{F}-semisimple.
\end{enumerate}
\end{proposition}

\noindent We proceed to compute Hodge invariants. We retain the notation of
Proposition \ref{something} and we write $D_{i,L}:=L\otimes_{L_{0}}D_{i}$ for
$i=1,2$ and $D_{\theta,L}:=L\otimes_{L_{0}}D_{\theta}$ for any $\theta\in
E^{\times}.$

\begin{proposition}
\label{t_H}The Hodge invariants of the filtered modules $D_{L},$ $D_{i,L}$ and
$D_{\theta,L}$ are%
\[
t_{H}(D_{L})=\sum\limits_{i\in I_{0}}k_{i},\ \ t_{H}(D_{1,L})=\sum
\limits_{\{i\in I_{0}\ :\ y_{i}=0\}}k_{i},\ t_{H}(D_{2,L})=\sum\limits_{\{i\in
I_{0}\ :\ x_{i}=0\}}k_{i}%
\]
and
\[
t_{H}(D_{\theta})=\noindent\sum\limits_{\{i\in J_{\vec{x}}\ \cap\ J_{\vec{y}%
}\ :\text{ }x_{i}\theta=y_{i}\}}k_{i}.
\]

\end{proposition}

\begin{proof}
The formula for $t_{H}(D_{L})$ follows immediately form Formula (\ref{t hodge}%
)$\ $since
\[
\text{dim}_{E}(E^{\mid\mathcal{S}_{L}\mid})f_{J}\left(  \vec{x}(1\otimes
\eta_{1})+f_{J_{\vec{y}}}\left(  1\otimes\eta_{2}\right)  \right)  =\mid J\mid
\]
for any $J\subset I_{0}$ (recall that $(x_{i},y_{i})\not =(0,0)$ for all $i$).
\noindent By definition,
\[
Fil^{j}(D_{2,L})=D_{2,L}\cap\text{Fil}^{j}(D_{L})
\]
for all $j.$ Let $1+w_{r-1}\leq j\leq w_{r}$ for some $1\leq r\leq t-1.$ We
have $\vec{t}(1\otimes\eta_{2})=\vec{\xi}\cdot f_{I_{r}}\cdot\vec{x}(\left(
1\otimes\eta_{1})+\vec{y}(1\otimes\eta_{2})\right)  $ if and only if $\vec
{\xi}\cdot\vec{x}\cdot f_{I_{r}}=\vec{0}$ and $\vec{\xi}\cdot\vec{y}\cdot
f_{I_{r}}=\vec{t}.$ For all $i\in I_{r}$ with $x_{i}\not =0$ we have $\xi
_{i}=0.$ If $x_{i}=0,$ then $y_{i}\not =0$ and as $\vec{\xi}$ varies in
$E^{\mid\mathcal{S}_{L}\mid}\ $the vector $\vec{\xi}\cdot\vec{y}\cdot
f_{I_{r}}$ can be any element of $f_{I_{r}\cap\ J_{\vec{x}}^{^{\prime}}%
}(E^{\mid\mathcal{S}_{L}\mid}),$ where $J_{\vec{x}}^{\prime}$ is the
complement of $J_{\vec{x}}$ in $I_{0}.$ Let $I_{r,\vec{x}}=I_{r}\cap
J_{\vec{x}}^{\prime}.$ For all $1+w_{r-1}\leq j\leq w_{r},\ $one has
Fil$^{j}(D_{2,L})=(E^{\mid\mathcal{S}_{L}\mid})f_{I_{r,\vec{x}}}(1\otimes
\eta_{2})$ and therefore
\[
\text{Fil}^{j}(D_{2,L})=\left\{
\begin{array}
[c]{l}%
\ \ D_{2,L}\text{\ \ \ \ \ \ \ \ \ if\ \ }j\leq0,\\
\left(  E^{\mid S_{L}\mid_{I_{i,\vec{x}}}}\right)  (1\otimes\eta_{2}),\text{
\ if\ \ }\\
1+w_{i-1}\leq j\leq w_{i},\ \text{for }i=0,1,...,t-1,\\
\ \ 0\text{\ \ \ \ \ \ \ \ \ \ \ \ \ \ if \ }j\geq1+w_{t-1}.
\end{array}
\right.
\]
Clearly $t_{H}(D_{2,L})=\sum\limits_{i=0}^{t-1}w_{i}(\mid I_{i,\vec{x}}%
\mid-\ I_{i+1,\vec{x}}\mid)$ (with $I_{t,\vec{x}}=\varnothing$). Since
\noindent$\mid I_{i,\vec{x}}\mid-\mid I_{i+1,\vec{x}}\mid=\#\{j\in I_{0}%
:k_{j}=w_{i}$ and $x_{j}=0\},$ we have%
\[
t_{H}(D_{2,L})=\sum\limits_{\{i\in I_{0}:\ x_{i}=0\}}k_{i}.
\]
The computation for $t_{H}(D_{1,L})\ $is identical.\ Last, for any $\theta\in
E^{\times},$
\[
\text{Fil}^{j}(D_{\theta})=D_{\theta}\cap\text{Fil}^{j}(D).
\]
Let $1+w_{r-1}\leq j\leq w_{r}$ for some $1\leq r\leq t-1$ and let $\vec
{t}(\eta_{1}+\theta\cdot\vec{1}\eta_{2})=\vec{\xi}\cdot f_{I_{r}}(\vec{x}%
\eta_{1}+\vec{y}\eta_{2})\in\ $Fil$^{j}(D_{\theta}).$ One easily sees that
$t_{i}\ $can be any elements of $E\ $as $\xi_{i}\ $varies in $E$ if and only
if $y_{i}=x_{i}\theta,$ and $t_{i}=0$ in any other case. Therefore
Fil$^{j}D_{\theta}=\left(  E^{\mid S_{L}\mid_{I_{r}(\theta)}}\right)
(\eta_{1}+\theta\cdot\vec{1}\eta_{2}),$ where \noindent$I_{r}(\theta
):=I_{r}\cap J_{\vec{x}}\cap J_{\vec{y}}\cap\{i\in I_{0}:x_{i}\theta
=y_{i}\ \}$ \noindent for all $1+w_{r-1}\leq j\leq w_{r}.$ This implies
\noindent$t_{H}(D_{\theta})=\sum\limits_{i=0}^{t-1}w_{i}\#\{i\in I_{0}%
:w_{i}=k_{i},\ x_{i}y_{i}\neq0\ $and$\ \theta=x_{i}^{-1}\cdot y_{i}%
\}=\sum\limits_{\{i\in J_{\vec{x}}\cap J_{\vec{y}}:\text{ }x_{i}\theta
=y_{i}\}}k_{i}.$
\end{proof}

\noindent For the Newton invariants of $D,$ $D_{i},$ and $D_{\theta}$ we have
the following proposition.

\begin{proposition}
If the diagonal entries of the matrix of $\varphi$ with respect to a standard
basis are $\alpha\cdot\vec{1}\ $and $\delta\cdot\vec{1},$ then $t_{N}%
(D)=efv_{p}(\alpha\delta),$ $\ t_{N}(D_{2})=efv_{p}(\delta),\ \ t_{N}(D_{1})$
\noindent$=efv_{p}(\alpha)$ and $t_{N}(D_{\theta})=efv_{p}(\alpha).$
\end{proposition}

\begin{proof}
Follows easily from Formula (\ref{t newton}) in the beginning of the section.
\end{proof}

\section{The weakly admissible rank two filtered
modules.\label{list of the weakly admissible}}

We summarize the results of the previous sections and list the rank two weakly
admissible filtered $\left(  \varphi,N,L/K,E\right)  $-modules. Before doing
so, we briefly digress to recall some well known facts about Galois types
(\cite[App.B]{CDT99}).

\subsection{Galois types}

Let $\rho:$ $G_{K}\rightarrow GL(V)$ be an $L$-semistable $n$-dimensional $E
$-representation of $G_{K},\ $as in the introduction. Let $W_{L}$ be the Weil
group of $L$ and $W_{K}$ the Weil group of $K.$ Recall that $W_{K}/W_{L}=$
Gal$(L/K).$ The Frobenius endomorphism $\varphi$ of $D_{st}^{L}(V)$ defines an
$E$-linear isomorphism
\[
\varphi:e_{\tau_{i+1}}D_{st}^{L}(V)\rightarrow e_{\tau_{i}}D_{st}^{L}(V),
\]
for each embedding $\tau_{i}\ $of $L_{0}\ $in $E.$ If $e_{K}$ is the absolute
ramification index $K,$ we define an $L_{0}$-linear action of $g\in W_{K}$ on
$D_{st}^{L}(V)$ given by $(g\operatorname{mod}W_{L})\circ\varphi
^{-\alpha(g)e_{K}},$ were the image of $g\ $in Gal$(\bar{k}_{K}/k_{K})\ $is
the $\alpha(g)$-th power of the $q_{K}$-th power map, with $k_{K}$ being the
residue field of $K\ $and $q_{K}$ its cardinality. Since $V$ is $L$%
-semistable, each component $e_{\tau_{i}}D_{st}^{L}(V)$ is an $E$-vector
spaces of dimension $n$ with an induced action of $(W_{K},N).$ Its isomorphism
class is independent of the choice of the embedding $\tau_{i}\ $(cf.
\cite[Lemme 2.2.1.2]{BM02}), and this unique isomorphism class is the
Weil-Deligne representation $WD(\rho)$ attached to $\rho.$

\begin{definition}
A Galois type of degree $2$ is an equivalence class of representations
$\tau:I_{K}\rightarrow GL_{2}(\bar{%
\mathbb{Q}%
}_{p})$ with open kernel which extend to $W_{K}.\ $We say that a
two-dimensional potentially semistable representation has Galois type $\tau$
if $WD(\rho)\mid_{I_{K}}\simeq\tau.\ $
\end{definition}

We have the following lemma.

\begin{lemma}
\label{types lemma}Assume that $p>2$ and let $\tau$ be a Galois type of degree
$2.$ Then $\tau$ has one of the following forms:

$(1)\ \tau\simeq\chi_{1}\mid_{I_{K}}\oplus\chi_{2}\mid_{I_{K}},$ where
$\chi_{1}$ and $\chi_{2}$ are characters of $W_{K}$ finite on $I_{K}; $

$(2)\ \tau\simeq$ Ind$_{W_{K^{\prime}}}^{W_{K}}(\chi)\mid_{I_{K}}\simeq
\chi\mid_{I_{K}}\oplus\chi^{h}\mid_{I_{K}},$ where $K^{\prime}$ is the
quadratic unramified extension of $K,\ \chi$ is a character of $W_{K^{\prime}%
}$ finite on $I_{K^{\prime}}$ which does not extend to $W_{K}, $ and $h$ a
generator of Gal$(K^{\prime}/K);$

$(3)\ \tau\simeq$ Ind$_{W_{K^{\prime}}}^{W_{K}}(\chi)\mid_{I_{K}},$ where
$K^{\prime}$ is a ramified quadratic extension of $K$ and $\chi$ a character
of $W_{K^{\prime}},$ finite on $I_{K^{\prime}},$ such that $\chi
\mid_{I_{K^{\prime}}}$which does not extend to $I_{K}.$
\end{lemma}

\noindent For Galois types we have the following three possibilities:

\begin{itemize}
\item $N\neq0$ and $\tau$ is a scalar (special or Steinberg case);

\item $N=0$ and $\tau$ as in (1) of Lemma $\ref{types lemma}$ (principal
series case);

\item $N=0$ and $\tau$ as in (2) or (3) of Lemma $\ref{types lemma}$
(supercuspidal case).\noindent
\end{itemize}

\noindent Notice that in the unramified supercuspidal case (Case (2) of Lemma
\ref{types lemma}), $\tau$ is reducible and the characters $\chi\mid_{I_{K}}$
and $\chi^{h}\mid_{I_{K}}$ are necessarily distinct, while in the ramified
supercuspidal case (Case (3) of Lemma \ref{types lemma}), $\tau$ is irreducible.

\noindent We now provide the list of rank two weakly admissible filtered
$\left(  \varphi,N,L/K,E\right)  $-modules and comment on the Galois type of
the corresponding potentially semistable representation, understanding that
the above mentioned terminology applies only in case that $p$ is odd, an
assumption not necessary in this paper.

Recall from Section \ref{stable filtr} that there is a right action of $G=$
Gal$(L/K)$ on $I_{0}$ defined by $i\cdot g:=\pi(g)(i),$ where $\pi$ is as in
Section \ref{action F}. This action has orbits $\mathcal{O}_{1},\mathcal{O}%
_{2},...,\mathcal{O}_{\nu},$ where $\nu=\left[  K:%
\mathbb{Q}
_{p}\right]  .$ Let $i_{j}$ be any fixed index in the orbit $\mathcal{O}_{j}$
for any $1\leq j\leq\nu,$ and choose any fixed pair $(x_{i_{j}},y_{i_{j}})\in
E\times E$ with $(x_{i_{j}},y_{i_{j}})\neq(0,0).\ $Assume that the labeled
Hodge-Tate weights are $(\{-k_{i},0\})_{\sigma_{i}},$ with $k_{i}$ non
negative integers.

\subsection{The F-semisimple, non-scalar case}

\noindent There exists an ordered basis

$\underline{\eta}=(\eta_{1},\eta_{2})$ of $D$ over $E^{\mid\mathcal{S}_{L_{0}%
}\mid}$ such that:

\begin{itemize}
\item The Frobenius endomorphism $\varphi$ of $D$ is given by $[\varphi
]_{\underline{\eta}}=$ diag$(\alpha\cdot\vec{1},\delta\cdot\vec{1})$ with
$\alpha,\delta\in E^{\times}$ and $\alpha^{f}\neq\delta^{f};$

\item The Galois action is given by $[g]_{\underline{\eta}}=$ diag$(\chi
_{1}(g)\cdot\vec{1},\chi_{2}(g)\cdot\vec{1})$ for some characters $\chi
_{i}:G\rightarrow E^{\times};$

\item The Galois-stable filtrations are equivalent to
\[%
\begin{array}
[c]{l}%
\ \text{Fil}^{j}(D_{L})=\left\{
\begin{array}
[c]{l}%
\ \ \ \ D_{L}\text{ \ if\ \ \ }j\,\leq0,\\
\left(  E^{\mid S_{L}\mid}\right)  \left(  \vec{x}(1\otimes\eta_{1})+\vec
{y}(1\otimes\eta_{2})\right)  \ \text{\ }\ \ \text{\ \ if\ }1\leq j\leq
w_{0},\\
\left(  E^{\mid S_{L}\mid_{I_{1}}}\right)  \left(  (\vec{x}1\otimes\eta
_{1})+\vec{y}(1\otimes\eta_{2})\text{\ }\right)  \ \text{\ \ if\ \ }%
1+w_{0}\leq j\leq w_{1},\\
\ \ \ \ \ \ \ \ \ \ \ \ \ \ \ \ \ \ \ \ \ \ \cdots\cdots\cdots\\
\left(  E^{\mid S_{L}\mid_{I_{t-1}}}\right)  \left(  \vec{x}(1\otimes\eta
_{1})+\vec{y}(1\otimes\eta_{2})\right)  \ \text{if\ \ }1+w_{t-2}\leq j\leq
w_{t-1},\\
\ \ \ \ \ \ 0\text{\ \ if \ \ }j\geq1+w_{t-1},
\end{array}
\right. \\
\\
\text{with\ }\vec{x}=\sum\limits_{j=1}^{\nu}\left\{  \sum\limits_{g\in
G}x_{i_{j}}\cdot\chi_{1}(g^{-1})\cdot e_{\pi(g)(i_{j})}\right\}  ,\ \ \vec
{y}=\sum\limits_{j=1}^{\nu}\left\{  \sum\limits_{g\in G}y_{i_{j}}\cdot\chi
_{2}(g^{-1})\cdot e_{\pi(g)(i_{j})}\right\}  ,\\
\\
\text{where the sets}\ I_{r}\ \text{are}\ \text{unions\ of\ }G\text{-orbits}%
\ \text{of}\ I_{0}\ \text{for\ all\ }r.
\end{array}
\]
$\ \ $
\end{itemize}

\subsubsection{The potentially crystalline
case\label{pot crystalline f semisimple}}

\begin{itemize}
\item The Frobenius-stable submodules are $0,\ D,$ $D_{1}=(E^{\mid
\mathcal{S}_{L_{0}}\mid})\eta_{1}$ and

$D_{2}=(E^{\mid\mathcal{S}_{L_{0}}\mid})\eta_{2};$

\item The filtered $(\varphi,L/K,E)$-module $D$ is weakly admissible if and
only if \noindent%
\begin{equation}%
\begin{array}
[c]{c}%
\text{(i)}\ efv_{p}(\alpha\delta)=\sum\limits_{i\in I_{0}}k_{i}\ \ \\
\\
\text{(ii)}\ efv_{p}(\alpha)\geq\sum\limits_{\{i\in I_{0}:\ y_{i}=0\}}%
k_{i}\ \text{and}\noindent\ \text{(iii)}\ efv_{p}(\delta)\geq\sum
\limits_{\{i\in I_{0}:\ x_{i}=0\}}k_{i},
\end{array}
\label{1}%
\end{equation}
$\ \ \ \ \ \ $
\end{itemize}

\noindent where $e$ is the absolute ramification index and $f$ the absolute
inertia degree of $L.$ Assuming that $D$ is weakly admissible,

\begin{enumerate}
\item It is irreducible if and only if both inequalities (ii) and (iii) in
(\ref{1}) are strict;

\item It is reducible, non-split if and only if exactly one of the
inequalities in (\ref{1}) is strict. If inequality (ii) is strict, the only
nontrivial weakly admissible submodule is $D_{2},$ while if inequality (iii)
is strict the only weakly admissible submodule is $D_{1};$

\item It is split-reducible if and only if $I_{0}^{+}\cap J_{\vec{x}}\cap
J_{\vec{y}}=\varnothing.$ \noindent The only nontrivial weakly admissible
submodules are $D_{1}\ $and $D_{2}.$
\end{enumerate}

\noindent The corresponding potentially crystalline representation is a
principal series.

\subsubsection{The potentially semistable, noncrystalline case}

In this case, there exists a basis $\underline{\eta}$ so that $\alpha
=p\delta\boldsymbol{\noindent}$ and $[N]_{\underline{\eta}}=\left(
\begin{array}
[c]{cc}%
\vec{0} & \vec{0}\\
\vec{1} & \vec{0}%
\end{array}
\right)  .\ $Moreover,

\begin{itemize}
\item The characters $\chi_{1}$ and $\chi_{2}$ are equal;

\item The submodules fixed by Frobenius and the monodromy are $0,$ $D$ and
$D_{2};$
\end{itemize}

The filtered $(\varphi,N,L/K)$-module$\ D$ is weakly admissible if and only if
\noindent%
\begin{equation}
2efv_{p}(\delta)+ef=\sum\limits_{i\in I_{0}}k_{i}\ \text{and}\ efv_{p}%
(\delta)\geq\sum\limits_{\{i\in I_{0}:\ x_{i}=0\}}k_{i}. \label{2}%
\end{equation}
\noindent Assuming that $D$ is weakly admissible, it is reducible, non-split
if and only if the inequality in (\ref{2}) is equality.$\ \noindent$In this
case, the only nontrivial weakly admissible submodule stable under Frobenius
and the monodromy is $D_{2}.$ In any other case $D$ is irreducible.

\noindent The corresponding potentially semistable representation is a special series.

\subsection{The \textnormal{F}-scalar case}

There exists an ordered basis $\underline{\eta}$ of $D$ over $E^{\mid
\mathcal{S}_{L_{0}}\mid}$ such that $[\varphi]_{\underline{\eta}}=$
diag$(\alpha\cdot\vec{1},\alpha\cdot\vec{1})$ with $\alpha\in E^{\times}.$

\begin{itemize}
\item The monodromy operator $N$ is trivial;

\item There exists a group homomorphism $\lambda:G\rightarrow GL_{2}(E)$ such that

$[g]_{\underline{\eta}}=\lambda(g)\cdot$diag$(\vec{1},\vec{1})$ for all $g\in
G;$

\item The Galois-stable filtrations are as in the non-\textnormal{F}-scalar case
with
\[
\vec{x}=\sum\limits_{j=1}^{\nu}\left\{  \sum\limits_{g\in G}x_{\pi(g)(i_{j}%
)}\cdot e_{\pi(g)(i_{j})}\right\}  ,\text{\ }\ \vec{y}=\sum\limits_{j=1}^{\nu
}\left\{  \sum\limits_{g\in G}y_{\pi(g)(i_{j})}\cdot e_{\pi(g)(i_{j}%
)}\right\}  ,
\]
where $\left(  x_{\pi(g)(i_{j})},y_{\pi(g)(i_{j})}\right)  =\left(  x_{i_{j}%
},y_{i_{j}}\right)  \cdot\lambda(g^{-1})$ for all $g\in G;$

\item The Frobenius-stable submodules are $0,\ D,$ $D_{1},$ $D_{2},$ with
$D_{1}$ and $D_{2}$ as in the previous cases, and $D_{\theta}=(E^{\mid
\mathcal{S}_{L_{0}}\mid})(\eta_{1}+\theta\cdot\vec{1}\eta_{2})$ for any
$\theta\in E^{\times}.$
\end{itemize}

\noindent For each $c\in E^{\times},$ let $k(c):=\sum\limits_{\{i\in
J_{\vec{x}}\cap J_{\vec{y}}:\ x_{i}^{-1}y_{i}=c\}}k_{i},$ where $x_{i}$ and
$y_{i}$ are the coordinates of the vectors $\vec{x}$ and $\vec{y}.\ $Let $k$
be the maximum of the integers $k(c).$ The filtered $\varphi$-module $D$ is
weakly admissible if and only if%
\begin{equation}%
\begin{array}
[c]{c}%
\text{(i)\ }2efv_{p}(\alpha)=\sum\limits_{i\in I_{0}}k_{i},\ \text{(ii)}%
\ efv_{p}(\alpha)\geq\sum\limits_{\{i\in I_{0}:\ y_{i}=0\}}k_{i},\ \\
\\
\text{(iii)}\ efv_{p}(\alpha)\geq\sum\limits_{\{i\in I_{0}:\ x_{i}=0\}}%
k_{i},\ \text{and}\ \text{(iv)}\ efv_{p}(\alpha)\geq k.
\end{array}
\label{3}%
\end{equation}

Assuming that $D$ is weakly admissible,

\begin{enumerate}
\item It is irreducible if and only if all inequalities (ii),\ (iii) and (iv)
in (\ref{3}) are strict.

\item It is reducible, non-split if and only if either exactly one of the
inequalities (ii) and (iii) is equality and inequality (iv) is strict, or both
inequalities (ii) and (iii) above are strict, inequality (iv) is equality and
the maximum is attained for precisely one constant $c.$ The only $\varphi
$-stable weakly admissible submodules are $D_{1},$ $D_{2}$ and $D_{c}$ respectively.

\item It is split-reducible if and only if either \noindent$x_{i}^{-1}y_{i}%
\ $is a constant $c$ for all $i\in I_{0}^{+}\cap J_{\vec{x}}\cap J_{\vec{y}}$
(including the case$\ I_{0}^{+}\cap J_{\vec{x}}\cap J_{\vec{y}}=\varnothing$
in which we define $c=0$) and one of the inequalities (ii) and (iii) above is
equality, or there exist two distinct constants $c_{1},$ $c_{2}$ such that
$k(c_{1})=k(c_{2}).$ The only weakly admissible submodules are $D_{1}$ and
$D_{c},$ or $D_{2}$ and $D_{c},\ $or $D_{c_{1}}$ and $D_{c_{2}}$ respectively,
and all these pairs of submodules are complementary in $D.$\noindent
\end{enumerate}

\noindent The corresponding potentially crystalline representation is
supercuspidal or principal series, depending on $\lambda.$

\subsection{The non-\textnormal{F}-semisimple case}

There exists an ordered basis $\underline{\eta}$ of $D$ over $E^{\mid
\mathcal{S}_{L_{0}}\mid}$ such that $[\varphi]_{\underline{\eta}}=\left(
\begin{array}
[c]{cc}%
\alpha\cdot\vec{1} & \vec{0}\\
\vec{1} & \alpha\cdot\vec{1}%
\end{array}
\right)  ,$ with $\alpha\in E^{\times}.$ In this case the monodromy operator
$N$ is trivial.

\begin{itemize}
\item The Galois action is given by $[g]_{\underline{\eta}}=$ diag$(\chi
(g)\cdot\vec{1},\chi(g)\cdot\vec{1})$ for some character $\chi:G\rightarrow
E^{\times},$ and the $G$-stable filtrations are as in the F-semisimple,
non-scalar case;

\item The Frobenius-fixed submodules are $0,$ $D,$ $D_{2};$
\end{itemize}

The filtered $\varphi$-module$\ D$ is weakly admissible if and only if
\noindent%
\begin{equation}
2efv_{p}(\alpha)+ef=\sum\limits_{i\in I_{0}}k_{i}\ \text{and}\ efv_{p}%
(\alpha)\geq\sum\limits_{\{i\in I_{0}:\ x_{i}=0\}}k_{i}. \label{4}%
\end{equation}
\noindent\noindent Assuming that $D$ is weakly admissible, it is reducible,
non-split if and only if the inequality in (\ref{4}) is equality.$\ \noindent
$In this case, the only nontrivial weakly admissible submodule is $D_{2}.$ In
any other case $D$ is irreducible.

\noindent The corresponding potentially crystalline representation is a
principal series.

\section{Isomorphism classes\label{isomorphism classes}}

\noindent\noindent Let $(D_{i},\varphi_{i},N_{i}),$ $i=1,2\,$be isomorphic
filtered $\left(  \varphi,N,L/K,E\right)  $-modules with labeled Hodge-Tate
weights $(\{-k_{\sigma},0\})_{\sigma},$ where $k_{\sigma}$ are non negative
integers. Let $\underline{\eta}^{i}=(\eta_{1}^{i},\eta_{2}^{i}),$ $i=1,2$ be
standard bases and let $h:D_{1}\rightarrow D_{2}$ be an isomorphism. We denote
by $[h]_{\underline{\eta}^{1}}^{\underline{\eta}^{2}}\ $the matrix of $h$ with
respect to the bases $\underline{\eta}^{i}$ and by $[h]_{1\otimes
\underline{\eta}^{1}}^{1\otimes\underline{\eta}^{2}} $ the matrix of
$h_{L}=1_{L\otimes_{%
\mathbb{Q}
_{p}}E}\otimes h$ with respect to the bases $1\otimes\underline{\eta}^{i}.$ If
all the weights $k_{\sigma}$ equal zero, compatibility of $h$ with the
filtrations holds trivially and the corresponding sections should be ignored.

\subsection{The F-semisimple, non-scalar case}

Let $[\varphi_{i}]_{\underline{\eta}^{i}}=$ diag$(\alpha_{i}\cdot\vec
{1},\delta_{i}\cdot\vec{1}),\ $with $\alpha_{i}^{f}\neq\delta_{i}^{f}\ $and
$\alpha_{i}=p\delta_{i}\neq0$ if the monodromy operators are nontrivial. In
the next proposition we determine when the isomorphism $h$ commutes with the
Frobenius operators. We write $Q=[h]_{\underline{\eta}^{1}}^{\underline{\eta
}^{2}}=\left(
\begin{array}
[c]{ll}%
\vec{a} & \vec{b}\\
\vec{c} & \vec{d}%
\end{array}
\right)  ,$ and by Section \ref{action F} it is clear that $\left(
[h_{L}]_{1\otimes\underline{\eta}^{1}}^{1\otimes\underline{\eta}^{2}}\right)
=Q^{\otimes e}=\left(
\begin{array}
[c]{ll}%
\vec{a}^{\otimes e} & \vec{b}^{\otimes e}\\
\vec{c}^{\otimes e} & \vec{d}^{\otimes e}%
\end{array}
\right)  =:\left(
\begin{array}
[c]{ll}%
\vec{a}_{1} & \vec{b}_{1}\\
\vec{c}_{1} & \vec{d}_{1}%
\end{array}
\right)  .$

\begin{proposition}
\label{h matrix}The isomorphism $h$ commutes with Frobenius endomorphisms if
and only if either

\begin{enumerate}
\item $\alpha_{1}^{f}=\alpha_{2}^{f}$ and $\delta_{1}^{f}=\delta_{2}^{f}, $ in
which case $[h]_{\underline{\eta}^{1}}^{\underline{\eta}^{2}}=$ diag$(a\cdot
\vec{a}_{0},d\cdot\vec{d}_{0}),$ where $\vec{a}_{0}=(1,\mu_{1},\mu_{1}%
^{2},...,\mu_{1}^{f-1}),$ $\vec{d}_{0}=(1,\mu_{2},\mu_{2}^{2},...,\mu
_{2}^{f-1}),$ with $\mu_{1}=\frac{\alpha_{1}}{\alpha_{2}},$ $\mu_{2}%
=\frac{\delta_{1}}{\delta_{2}}$ and $a,d\in E^{\times},\ $or

\item $\alpha_{1}^{f}=\delta_{2}^{f}$ and $\delta_{1}^{f}=\alpha_{2}^{f},$ in
which case $[h]_{\underline{\eta}^{1}}^{\underline{\eta}^{2}}$ $=\left(
\begin{array}
[c]{ll}%
\ \ \vec{0} & b\cdot\vec{b}_{0}\\
c\cdot\vec{c}_{0} & \ \ \ \vec{0}%
\end{array}
\right)  ,$ where $\vec{b}_{0}=(1,\xi_{1},\xi_{1}^{2},...,\xi_{1}^{f-1}),$
$\vec{c}_{0}=(1,\xi_{2},\xi_{2}^{2},...,\xi_{2}^{f-1}),$ with $\xi_{1}=$
$\frac{\delta_{1}}{\alpha_{2}},\ \xi_{2}=\frac{\alpha_{1}}{\delta_{2}}$ and
$b,c\in E^{\times}.$
\end{enumerate}
\end{proposition}

\begin{proof}
We need $([\varphi_{2}]_{\underline{\eta}^{2}})\cdot\varphi(Q)=Q\cdot
([\varphi_{1}]_{\underline{\eta}^{1}})$, or equivalently $\alpha_{1}\vec
{a}=\alpha_{2}\varphi(\vec{a}),\delta_{1}\vec{b}=\alpha_{2}\varphi(\vec
{b}),\alpha_{1}\vec{c}=\delta_{2}\varphi(\vec{c})\ $and$\ \delta_{1}\vec
{d}=\delta_{2}\varphi(\vec{d}).$If $\alpha_{1}^{f}\not \in \{\alpha_{2}%
^{f},\delta_{2}^{f}\},$ then Lemma \ref{norm lemma} implies $\vec{a}=\vec
{c}=\vec{0}$ a contradiction. Hence $\alpha_{1}^{f}\in\{\alpha_{2}^{f}%
,\delta_{2}^{f}\},$ and similarly $\delta_{1}^{f}\in\{\alpha_{2}^{f}%
,\delta_{2}^{f}\}.$ Since $\alpha_{i}^{f}\neq\delta_{i}^{f}$ for $i=1,2$ we
have the following cases: Case\ (1). If $\alpha_{1}^{f}=\alpha_{2}^{f}$ and
$\delta_{1}^{f}=\delta_{2}^{f}.$ By Lemma \ref{norm lemma}, $Q=$ diag$(\vec
{a},\vec{d}),$ where $\vec{a}=a(1,\mu_{1},\mu_{1}^{2},...,\mu_{1}^{f-1}),$
$\vec{d}=d(1,\mu_{2},\mu_{2}^{2},...,\mu_{2}^{f-1})$ with $\mu_{1}%
=\frac{\alpha_{1}}{\alpha_{2}},$ $\mu_{2}=\frac{\delta_{1}}{\delta_{2}}$ and
$a,d\in E^{\times}.\ $Case\ (2). If $\alpha_{1}^{f}=\delta_{2}^{f}$ and
$\delta_{1}^{f}=\alpha_{2}^{f}.$ Arguing as in Case (1), $Q=\left(
\begin{array}
[c]{ll}%
\vec{0} & \vec{b}\\
\vec{c} & \vec{0}%
\end{array}
\right)  ,$ with $\vec{b}=b(1,\xi_{1},\xi_{1}^{2},...,\xi_{1}^{f-1}),$
$\vec{c}=c(1,\xi_{2},\xi_{2}^{2},...,\xi_{2}^{f-1}),$ where $\xi_{1}=$
$\frac{\delta_{1}}{\alpha_{2}},\ \xi_{2}=\frac{\alpha_{1}}{\delta_{2}}$ and
$b,c\in E^{\times}.$
\end{proof}

\noindent We now determine when $h$ commutes with the monodromy operators.

\begin{proposition}
\ \label{commutativity with monodromies}The isomorphism $h$ commutes with the
monodromy operators if and only if either both the monodromies are trivial or
the matrix $[h]_{\underline{\eta}^{1}}^{\underline{\eta}^{2}}\ $is as in Case
(1) of Proposition \ref{h matrix}, $a=d\ $and $\alpha_{1}\delta_{2}=\alpha
_{2}\delta_{1}.$
\end{proposition}

\begin{proof}
Clearly the monodromy operator of one of the filtered modules is trivial if
and only if the monodromy operator of the other is.\ The monodromy operators
commute with $h$ if and only if $\left(  [h]_{\underline{\eta}^{1}%
}^{\underline{\eta}^{2}}\right)  [N_{1}]_{\underline{\eta}^{1}}=[N_{2}%
]_{\underline{\eta}^{2}}\left(  [h]_{\underline{\eta}^{1}}^{\underline{\eta
}^{2}}\right)  .$ The proposition follows\ by a straightforward computation
using Corollary \ref{simplify N} and Proposition \ref{h matrix}.
\end{proof}

\begin{proposition}
\noindent Let $[g]_{\underline{\eta}^{1}}=$ diag$(\chi_{1}(g)\cdot\vec{1}%
,\chi_{2}(g)\cdot\vec{1})$ and $[g]_{\underline{\eta}^{2}}=$ diag$(\psi
_{1}(g)\cdot\vec{1},\psi_{2}(g)\cdot\vec{1}).$

$(1)$\ If the matrix of $h$ is as in Case $(1)$ of Proposition \ref{h matrix},
then $h$ commutes with the Galois actions if and only if $\chi_{1}(g)=\mu
_{1}^{n(g)}\psi_{1}(g)$ and $\chi_{2}(g)=\mu_{2}^{n(g)}\psi_{2}(g)$ for all
$g\in G.$

$(2)$ If the matrix of $h$ is as in Case $(2)$ of Proposition \ref{h matrix},
then $h$ commutes with the Galois actions if and only if $\chi_{1}(g)=\xi
_{2}^{n(g)}\psi_{2}(g)$ and $\chi_{2}(g)=\xi_{1}^{n(g)}\psi_{1}(g)$ for all
$g\in G.$
\end{proposition}

\begin{proof}
A straightforward computation, using that the Galois actions commutes with $h$
if and only if $\left(  [h]_{\underline{\eta}^{1}}^{\underline{\eta}^{2}%
}\right)  [g]_{\underline{\eta}^{1}}=[g]_{\underline{\eta}^{2}}{}\left(
^{g}[h]_{\underline{\eta}^{1}}^{\underline{\eta}^{2}}\right)  .$
\end{proof}

\subsubsection{Compatibility with the
filtrations\label{compatibility with filtrations}}

Throughout this section we assume that at least one weight $k_{\sigma}$ is
positive. Suppose that for $i=1,2$ we have%

\[
\text{Fil}^{j}(D_{i,L})=\left\{
\begin{array}
[c]{l}%
\ \ \ \ \ \ \ \ \ \ \ \ D_{i,L}\text{ \ if\ \ \ }j\,\leq0,\\
\left(  E^{\mid S_{L}\mid_{I_{r}}}\right)  \left(  \vec{x}_{i}(1\otimes
\eta_{1})+\vec{y}_{i}(1\otimes\eta_{2})\right)  \text{\ if}\\
1+w_{r-1}\leq j\leq w_{r},\ \text{for}\ r=0,...,t-1,\\
\ \ \ \ \ \ \ \ \ \ \ \ \ 0\text{\ \ if \ \ }j\geq1+w_{t-1},
\end{array}
\right.
\]
We need
\begin{equation}
\noindent h_{L}(\text{Fil}^{j}D_{1,L})=\text{Fil}^{j}D_{2,L} \label{star}%
\end{equation}
for all $j$ and we\ have\ the\ following\ cases:$\ $\noindent(1)\textit{\ }If
$Q=$ diag$(\vec{a},\vec{d})$ is as in Case (1) of Proposition \ref{h matrix},
let $Q^{\otimes e}=$ diag$(\vec{a}_{1},\vec{d}_{1}),$ where $\vec{a}_{1}%
=\vec{a}^{\otimes e}$ and $\vec{d}_{1}=\vec{d}^{\otimes d}.$ Since $h_{L} $ is
$\left(  E^{\mid S_{L}\mid}\right)  $-linear. Condition $\noindent$%
(\ref{star}) is equivalent to
\[
\left(  E^{\mid S_{L}\mid}\right)  (f_{J_{\vec{x}_{1}}}\cdot\vec{x}_{1}%
\cdot\vec{a}_{1}(1\otimes\eta_{1}^{1})+f_{J_{\vec{y}_{1}}}\cdot\vec{d}%
_{1}(1\otimes\eta_{2}^{1}))=\left(  E^{\mid S_{L}\mid}\right)  (f_{J_{\vec
{x}_{2}}}\cdot\vec{x}_{2}\cdot((1\otimes\eta_{1}^{2})+f_{J_{\vec{y}_{2}}%
}((1\otimes\eta_{2}^{2})),
\]
and the latter equivalent to the system of equations
\begin{equation}
\ \ \ \text{(i)}\ \ \ \ \left\{
\begin{array}
[c]{c}%
f_{J_{\vec{x}_{1}}}\cdot\vec{a}_{1}\cdot\vec{x}_{1}=\vec{t}\cdot f_{J_{\vec
{x}_{2}}},\\
f_{J_{\vec{y}_{1}}}\cdot\vec{d}_{1}\cdot\vec{x}_{2}=\vec{t}\cdot f_{J_{\vec
{y}_{2}}},
\end{array}
\right\}  \ \ \ \ \text{and\ \ \ (ii)\ \ \ }\left\{
\begin{array}
[c]{c}%
f_{J_{\vec{x}_{2}}}=f_{J_{\vec{x}_{1}}}\cdot\vec{t}_{1}\cdot\vec{a}_{1},\\
f_{J_{\vec{y}_{2}}}=f_{J_{\vec{y}_{1}}}\cdot\vec{t}_{1}\cdot\vec{d}_{1},
\end{array}
\right\}  \ \label{5}%
\end{equation}
for some vectors $\vec{t},$ $\vec{t}_{1}\in E^{\mid S_{L}\mid}.$ We easily see
that (\ref{5}) implies$\ $%
\[
f_{J_{\vec{x}_{1}}\cap J_{\vec{y}_{2}}}\cdot\vec{a}_{1}\cdot\vec{x}%
_{1}=f_{J_{\vec{x}_{2}}\cap J_{\vec{y}_{1}}}\cdot\vec{d}_{1}\cdot\vec{x}_{2}.
\]
Since $\vec{a}_{1}\in\left(  E^{\times}\right)  ^{\mid S_{L}\mid},$ the first
equation of (\ref{5})(i) implies that $J_{\vec{x}_{1}}\subset J_{\vec{x}_{2}}$
and the first equation of (\ref{5})(ii) that $J_{\vec{x}_{2}}\subset
J_{\vec{x}_{1}},$ therefore $J_{\vec{x}_{1}}=J_{\vec{x}_{2}}.$

\noindent Similarly, since $\vec{d}_{1}\in\left(  E^{\times}\right)  ^{\mid
S_{L}\mid},$ we have $J_{\vec{y}_{1}}=J_{\vec{y}_{2}}.$ Conversely, if the
equations%
\[
J_{\vec{x}_{1}}=J_{\vec{x}_{2}};\ J_{\vec{y}_{1}}=J_{\vec{y}_{2}}%
\ \text{and}\ f_{J_{\vec{x}_{1}}\cap J_{\vec{y}_{2}}}\cdot\vec{a}_{1}\cdot
\vec{x}_{1}=f_{J_{\vec{x}_{2}}\cap J_{\vec{y}_{1}}}\cdot\vec{d}_{1}\cdot
\vec{x}_{2}%
\]
hold, then it is easy to see that the system of equations (\ref{5}) has
solutions in $\vec{t}$ and $\vec{t}_{1}.$ Hence, $h$ preserves the filtrations
if and only if
\begin{equation}
J_{\vec{x}_{1}}=J_{\vec{x}_{2}};\ J_{\vec{y}_{1}}=J_{\vec{y}_{2}}\ \text{and
}f_{J_{\vec{x}_{1}}\cap J_{\vec{y}_{1}}}\cdot\vec{a}_{1}\cdot\vec{x}%
_{1}=f_{J_{\vec{x}_{2}}\cap J_{\vec{y}_{2}}}\cdot\vec{d}_{1}\cdot\vec{x}_{2}
\label{6}%
\end{equation}
We have the following subcases:

\noindent(a)\ When the monodromies are trivial: In this case, the third
equation in (\ref{6}) can be replaced by
\begin{equation}
f_{J_{\vec{x}}\cap J_{\vec{y}}}\cdot\left(  \vec{a}_{0}\right)  ^{\otimes
e}\cdot\vec{x}_{1}=f_{J_{\vec{x}}\cap J_{\vec{y}}}\cdot(\vec{d}_{0})^{\otimes
e}\cdot\vec{x}_{2}\ \text{in\ the\ projective\ space\ }\mathbb{P}^{m-1}(E),
\label{n=0 or otherwise 1}%
\end{equation}
where $\vec{a}_{0}=(1,\mu_{1},\mu_{1}^{2},...,\mu_{1}^{f-1})$ and $\vec{d}%
_{0}=(1,\mu_{2},\mu_{2}^{2},...,\mu_{2}^{f-1}).$

\noindent Conversely, if $\alpha_{1}^{f}=\alpha_{2}^{f},$ $\delta_{1}%
^{f}=\delta_{2}^{f}$ and equation (\ref{n=0 or otherwise 1}) holds, then
(after scaling one of the vectors $\vec{a}_{0}$ or $\vec{d}_{0}$ if necessary)
$Q=\left(  [h]_{\bar{\eta}^{1}}^{\underline{\eta}^{2}}\right)  =$
diag$(\vec{a}_{0},\vec{d}_{0})$ defines an isomorphism of filtered $\left(
\varphi,N,L/K,E\right)  $-modules $h:(D_{1},\varphi_{1})\rightarrow
(D_{2},\varphi_{2}).$

\noindent(b)\ When the monodromies are nontrivial: By Proposition
\ref{commutativity with monodromies} we have $\vec{a}=\vec{d}$ and (\ref{6})
is equivalent to
\begin{equation}
J_{\vec{x}_{1}}=J_{\vec{x}_{2}};\ J_{\vec{y}_{1}}=J_{\vec{y}_{2}%
}\ \text{and\ }f_{J_{\vec{x}}\cap J_{\vec{y}}}\cdot\vec{x}_{1}=f_{J_{\vec{x}%
}\cap J_{\vec{y}}}\cdot\vec{x}_{2}. \label{tralala}%
\end{equation}
Conversely, if $\alpha_{1}^{f}=\alpha_{2}^{f},$ $\delta_{1}^{f}=\delta_{2}%
^{f},$ and $\alpha_{1}\delta_{2}=\alpha_{2}\delta_{1},$ if the monodromy
operators are non-trivial, and if equations (\ref{tralala}) hold, then the
$E^{\mid S_{L_{0}}\mid}$-linear map

\noindent$h:(D_{1},\varphi_{1})\rightarrow(D_{2},$ $\varphi_{2})$ defined by
$Q=[h]_{\bar{\eta}^{1}}^{\underline{\eta}^{2}}=$ diag$(\vec{a}_{0},\vec{a}%
_{0})$ is an isomorphism of filtered $\left(  \varphi,N,L/K,E\right)  $-modules.

\noindent(2)\ If $Q=\left(
\begin{array}
[c]{cc}%
\vec{0} & \vec{b}\\
\vec{c} & \vec{0}%
\end{array}
\right)  ,$ then both the monodromy operators are zero. Arguing before we see
that $h_{L}\ $preserves the filtrations if and only if%
\begin{equation}%
\begin{array}
[c]{c}%
J_{\vec{x}_{1}}=J_{\vec{y}_{2}};\ J_{\vec{y}_{1}}=J_{\vec{x}_{2}}%
\ \text{and}\\
\\
\ f_{J_{\vec{x}_{1}}\cap J_{\vec{y}_{1}}}\cdot(\vec{b}_{0})^{\otimes
e}=f_{J_{\vec{y}_{2}}\cap J_{\vec{x}_{2}}}\cdot\left(  \vec{c}_{0}\right)
^{\otimes e}\cdot\vec{x}_{1}\cdot\vec{x}_{2}\ \text{in\ }\mathbb{P}^{m-1}(E).
\end{array}
\label{lalala}%
\end{equation}
Conversely, if $\alpha_{1}^{f}=\delta_{2}^{f},$ $\delta_{1}^{f}=\alpha_{2}%
^{f}$ and equations (\ref{lalala}) hold, then the $E^{\mid S_{L_{0}}\mid}%
$-linear map $h:(D_{1},\varphi_{1})\rightarrow(D_{2},\varphi_{2})$ defined by
$Q=\left(  [h]_{\bar{\eta}^{1}}^{\underline{\eta}^{2}}\right)  =\left(
\begin{array}
[c]{ll}%
\vec{0} & \vec{b}_{0}\\
\vec{c}_{0} & \ \vec{0}%
\end{array}
\right)  $ (after scaling one of the vectors $\vec{b}_{0}$ or $\vec{c}_{0}$ if
necessary) is an isomorphism of filtered $\left(  \varphi,N,L/K,E\right)  $-modules.

\subsection{The F-scalar case\label{10}}

Suppose that
\[
\lbrack\varphi_{i}]_{\underline{\eta}^{i}}=\text{diag}(\alpha_{i}\cdot\vec
{1},\alpha_{i}\cdot\vec{1})\ \text{and\ }[g]_{\underline{\eta}^{i}}%
=\lambda_{i}(g)\cdot\text{diag}(\vec{1},\vec{1})
\]
for some group homomorphisms $\lambda_{i}:G\rightarrow GL_{2}(E),$ $i=1,2.$
Arguing as in the non-F-scalar case, one easily sees that an isomorphism $h$
commuting with Frobenius exists if and only if $\alpha_{1}^{f}=\alpha_{2}%
^{f}.$ Then, $Q=[h]_{\underline{\eta}^{1}}^{\underline{\eta}^{2}}=R\cdot
$diag$(\vec{1},\vec{1})$ for some $R\in GL_{2}(E),$ and $h$ commutes with the
Galois action if and only if $\lambda_{2}(g)=R\lambda_{1}(g)R^{-1}$ for all
$g\in G.$ Let $R=\left(
\begin{array}
[c]{cc}%
a & b\\
c & d
\end{array}
\right)  .$ Since $h_{L}$ is an $E^{\mid S_{L}\mid}$-linear isomorphism, it
preserves the filtrations if and only if $h_{L}\left(  \text{Fil}^{1}%
D_{1,L}\right)  =$ Fil$^{1}D_{2,L},$ or equivalently
\[
E^{\mid S_{L}\mid}\left(  a\cdot\vec{x}_{1}+b\cdot f_{J_{\vec{y}_{1}}}\right)
=E^{\mid S_{L}\mid}\vec{x}_{2}\ \text{and }E^{\mid S_{L}\mid}\left(
c\cdot\vec{x}_{1}+d\cdot f_{J_{\vec{y}_{1}}}\right)  =E^{\mid S_{L}\mid
}f_{J_{\vec{y}_{2}}}%
\]
which we write in assorted form as
\begin{equation}
\left(  E^{\mid S_{L}\mid}\right)  \left(  \vec{x}_{1},\ f_{J_{\vec{y}_{1}}%
}\right)  \cdot\left(  R\cdot\text{diag}(\vec{1},\vec{1})\right)  =\left(
\ E^{\mid S_{L}\mid}\right)  \left(  \vec{x}_{2},\ f_{J_{\vec{y}_{2}}}\right)
. \label{7}%
\end{equation}
Conversely, if $\alpha_{1}^{f}=\alpha_{2}^{f},$ if there exists some $R\in
GL_{2}(E)\ $such that

\noindent$\lambda_{2}(g)=R\lambda_{1}(g)R^{-1}$ for all $g\in G$ and (\ref{7})
holds, then the $E^{\mid S_{L_{0}}\mid}$-linear map
\[
h:D_{1}\rightarrow D_{2}\ \text{defined by }[h]_{\underline{\eta}^{1}%
}^{\underline{\eta}^{2}}=R\cdot\text{diag}(\vec{1},\vec{1})
\]
is an isomorphism of filtered $\left(  \varphi,N,L/K,E\right)  $-modules.

\subsection{The non-F-semisimple case}

Let
\[
\lbrack\varphi_{i}]_{\underline{\eta}^{i}}=\left(
\begin{array}
[c]{cc}%
\alpha_{i}\cdot\vec{1} & \vec{0}\\
\vec{1} & \alpha_{i}\cdot\vec{1}%
\end{array}
\right)  \ \text{and }[g]_{\underline{\eta}^{i}}=\text{diag}(\chi_{i}%
(g)\cdot\vec{1},\chi_{i}(g)\cdot\vec{1})
\]
for some characters $\chi_{i}:G\rightarrow E^{\times}.$ Let $Q=[h]_{\underline
{\eta}^{1}}^{\underline{\eta}^{2}}=\left(
\begin{array}
[c]{cc}%
\vec{a} & \vec{b}\\
\vec{c} & \vec{d}%
\end{array}
\right)  .$

\noindent The isomorphism $h$ commutes with the Frobenius endomorphisms if and
only if
\begin{equation}
([\varphi_{2}]_{\underline{\eta}^{2}})\cdot\varphi(Q)=Q\cdot([\varphi
_{1}]_{\underline{\eta}^{1}}). \label{8}%
\end{equation}
This implies that $Nm_{\varphi}([\varphi_{2}]_{\underline{\eta}^{2}})\cdot
Q=Q\cdot$ $Nm_{\varphi}([\varphi_{1}]_{\underline{\eta}^{1}}),$ and this
combined with Lemma \ref{norm lemma} that $\alpha_{1}^{f}=\alpha_{2}%
^{f},\ \vec{b}=\vec{0}$ and $\vec{a}=\vec{d}=a\cdot\left(  1,\frac{\alpha_{2}%
}{\alpha_{1}},\left(  \frac{\alpha_{2}}{\alpha_{1}}\right)  ^{2},...,\left(
\frac{\alpha_{2}}{\alpha_{1}}\right)  ^{f-1}\right)  $ for some $a\in
E^{\times}.$ Then by equation (\ref{8}), the coordinates of $\vec{c}$ satisfy%
\[
c_{i}=\mu_{1}^{i}\left\{  (c_{0}-a\mu_{1}^{-1}+a)-\sum\limits_{j=1}%
^{i-1}\left(  \mu_{1}^{-2j-1}-\mu_{1}^{-2j}\right)  \right\}  \ \text{for
}i=1,2,...,f-1,
\]
where $c_{0}\in E\ $is arbitrary. Arguing as in Section
\ref{compatibility with filtrations} we see that $h$ is preserves the
filtrations if and only if
\begin{equation}
J_{\vec{x}_{1}}=J_{\vec{x}_{2}}\ \text{and}\ f_{J_{\vec{x}}}\cdot\vec{x}%
\cdot\vec{x}_{1}\cdot\vec{c}^{\otimes{e}}=\left(  f_{J_{\vec{x}}\cap
J_{\vec{y}_{1}}}\cdot\vec{x}-f_{J_{\vec{x}}\cap J_{\vec{y}_{1}}}\cdot\vec
{x}_{1}\right)  \cdot\vec{a}^{\otimes{e}}. \label{hello!}%
\end{equation}
It is straightforward to see that $h$ commutes with the Galois actions if and
only if $\chi_{1}(g)=\mu_{1}^{n(g)}\cdot\chi_{2}(g)\ $and $^{g}\vec{c}=\mu
_{1}^{n(g)}\cdot\vec{c}$ for all $g.\ $The latter equation holds\ if and only
if either $\alpha_{1}=\alpha_{2},$ or $\sum\limits_{j=0}^{n(g)-1}\left(
\frac{\alpha_{2}}{\alpha_{1}}\right)  ^{2j}=0$ for all $g\in G.$ Conversely,
assume that $\alpha_{1}^{f}=\alpha_{2}^{f}$ and $\sum\limits_{j=0}%
^{n(g)-1}\left(  \frac{\alpha_{2}}{\alpha_{1}}\right)  ^{2j}=0$ for all $g\in
G,$ in case that $\alpha_{1}\neq\alpha_{2}.$ In addition, assume that
$\chi_{1}(g)=\mu_{1}^{n(g)}\cdot\chi_{2}(g)$ for all $g.$ If the first two
equations in (\ref{hello!}) hold and there exist $a\in E^{\times}$ and
$c_{0}\in E$ such that the third equation in (\ref{hello!}) holds, then the
$E^{\mid S_{L_{0}}\mid}$-linear map $h:D_{1}\rightarrow D_{2}\ $defined by
$[h]_{\underline{\eta}^{1}}^{\underline{\eta}^{2}}=\left(
\begin{array}
[c]{cc}%
\vec{a} & \vec{0}\\
\vec{c} & \vec{a}%
\end{array}
\right)  $ is an isomorphism of filtered $(\varphi,L/K,E)$-modules. We now
list the isomorphism classes of rank two filtered $\left(  \varphi
,N,L/K,E\right)  $-modules.

\subsection{\textbf{The list of isomorphism classes\label{last list}}}

Let $(D_{i},\varphi_{i},N_{i},L/K,E)$ be filtered modules with labeled
Hodge-Tate weights $(\{-k_{\sigma},0\})_{\sigma},$ with $k_{\sigma}$ non
negative integers. Let $\underline{\eta}^{i},$ $i=1,2,\ $be standard bases,
and suppose that the filtrations are given by
\[
\text{Fil}^{j}(D_{i,L})=\left\{
\begin{array}
[c]{l}%
\ \ \ \ \ \ \ \ \ \ \ \ D_{i,L}\text{ \ if\ \ \ }j\,\leq0,\\
\left(  E^{\mid S_{L}\mid_{I_{r}}}\right)  \left(  \vec{x}_{i}(1\otimes
\eta_{1}^{i})+\vec{y}_{i}(1\otimes\eta_{2}^{i})\right)  \text{\ if}\\
1+w_{r-1}\leq j\leq w_{r},\ \text{for}\ r=0,...,t-1,\\
\ \ \ \ \ \ \ \ \ \ \ \ \ 0\text{\ \ if \ \ }j\geq1+w_{t-1},
\end{array}
\right.
\]
for some vectors $\vec{x}_{i},\vec{y}_{i}\in E^{\mid S_{L}\mid}$ whose
coordinates do not vanish simultaneously. Throughout this section, any
equation involving the sets $J_{\vec{x}}$ and $J_{\vec{y}}$ should be ignored
if all the weights $k_{\sigma}$ equal zero. Recall the definition of $n(g)$
from Section \ref{action F_0}.

\subsubsection{The F-semisimple case}

Let$\ [\varphi_{i}]_{\underline{\eta}^{i}}=$ diag$(\alpha_{i}\cdot\vec
{1},\ \delta_{i}\cdot\vec{1})\ $with $\alpha_{i},\delta_{i}\in E^{\times}%
\ $such that $\alpha_{i}^{f}\neq\delta_{i}^{f}$ and $[g]_{\underline{\eta}%
^{1}}=$ diag$(\chi_{1}(g)\cdot\vec{1},\chi_{2}(g)\cdot\vec{1}),$
$[g]_{\underline{\eta}^{2}}=$ diag$(\psi_{1}(g)\cdot\vec{1},\psi_{2}%
(g)\cdot\vec{1})$ for some $E^{\times}$-valued characters $\chi_{i}\ $and
$\psi_{i}\ $of $G=$ Gal$(L/K).$ When the monodromy operators are nontrivial,
the bases are chosen so that $\alpha_{i}=p\delta_{i}$ and $[N_{i}%
]_{\underline{\eta}^{i}}=\left(
\begin{array}
[c]{cc}%
\vec{0} & \vec{0}\\
\vec{1} & \vec{0}%
\end{array}
\right)  .$

\subsubsection{The potentially crystalline
case\label{filtrations classif potentially crystalline}}

If both the monodromy operators are trivial, then $(D_{1},\varphi
_{1},L/K,E)\simeq(D_{2},\varphi_{2},L/K,E)$ if and only if either%

\[
\left\{
\begin{array}
[c]{c}%
\alpha_{1}^{f}=\alpha_{2}^{f}\\
\delta_{1}^{f}=\delta_{2}^{f}%
\end{array}
\right\}  ,\ \ \ \ \left\{
\begin{array}
[c]{c}%
J_{\vec{x}_{1}}=J_{\vec{x}_{2}}\\
J_{\vec{y}_{1}}=J_{\vec{y}_{2}}%
\end{array}
\right\}  ,~\ \ \ \left\{
\begin{array}
[c]{c}%
\chi_{1}(g)=\mu_{1}^{n(g)}\psi_{1}(g)\\
\chi_{2}(g)=\mu_{2}^{n(g)}\psi_{2}(g)\
\end{array}
\right\}
\]
for\ all$\ g\in G\ $and $~$%
\[
\vec{a}\cdot f_{J_{\vec{x}_{1}}\cap J_{\vec{y}_{1}}}\cdot\vec{x}_{1}=\vec
{d}\cdot f_{J_{\vec{x}_{2}}\cap J_{\vec{y}_{2}}}\cdot\vec{x}_{2}%
\ \text{in}\ \mathbb{P}^{m-1}(E),
\]
$\noindent$with$\ \vec{a}$\noindent$=\left(  1,\mu_{1},\mu_{1}^{2},...,\mu
_{1}^{f-1}\right)  ^{\otimes e}\ $and $\vec{d}=\left(  1,\mu_{2},\mu_{2}%
^{2},...,\mu_{2}^{f-1}\right)  ^{\otimes e},$ where $\mu_{1}=\frac{\alpha_{1}%
}{\alpha_{2}}$ and $\mu_{2}=\frac{\delta_{1}}{\delta_{2}},$ or%

\[
\left\{
\begin{array}
[c]{c}%
\alpha_{1}^{f}=\delta_{2}^{f}\\
\delta_{1}^{f}=\alpha_{2}^{f}%
\end{array}
\right\}  ,\ \ \ \ \left\{
\begin{array}
[c]{c}%
J_{\vec{x}_{1}}=J_{\vec{y}_{2}}\\
J_{\vec{y}_{1}}=J_{\vec{x}_{2}}%
\end{array}
\right\}  ,\ \ \ \ \left\{
\begin{array}
[c]{c}%
\chi_{1}(g)=\xi_{2}^{n(g)}\psi_{2}(g)\\
\chi_{2}(g)=\xi_{1}^{n(g)}\psi_{1}(g)\
\end{array}
\right\}
\]
for\ all\ $g\in G\ $and
\[
\vec{b}\cdot f_{J_{\vec{x}_{1}}\cap\ J_{\vec{y}_{1}}}=\vec{c}\cdot
f_{J_{\vec{x}_{1}}\cap\ J_{\vec{y}_{1}}}\cdot\vec{x}_{1}\cdot\vec{x}%
_{2}\ \text{in}\ \mathbb{P}^{m-1}(E)\noindent,
\]
with $\noindent$\noindent$\vec{b}=\left(  1,\xi_{1},\xi_{1}^{2},...,\xi
_{1}^{f-1}\right)  ^{\otimes e}$ and $\vec{c}=\left(  1,\xi_{2},\xi_{2}%
^{2},...,\xi_{2}^{f-1}\right)  ^{\otimes e},$ where \noindent$\xi_{1}=$
$\frac{\delta_{1}}{\alpha_{2}}$ and $\xi_{2}=\frac{\alpha_{1}}{\delta_{2}}.$

\subsubsection{The potentially semistable, noncrystalline case}

If both the monodromies are nontrivial, then \noindent$(D_{1},\varphi
_{1},N_{1},L/K,E)\simeq(D_{2},\varphi_{2},N_{2},L/K,E)$ if and only if%
\[
\left\{
\begin{array}
[c]{c}%
\alpha_{1}^{f}=\alpha_{2}^{f}\\
\alpha_{1}\delta_{2}=\alpha_{2}\delta_{1}%
\end{array}
\right\}  ,\ \ \ \ \left\{
\begin{array}
[c]{c}%
J_{\vec{x}_{1}}=J_{\vec{x}_{2}}\\
J_{\vec{y}_{1}}=J_{\vec{y}_{2}}%
\end{array}
\right\}  ,\ \ \ \left\{
\begin{array}
[c]{c}%
\chi_{1}(g)=\mu_{1}^{n(g)}\psi_{1}(g)\ \text{for\ all}\ g\in G\ \text{and}\\
f_{J_{\vec{x}_{1}}\cap\ J_{\vec{y}_{1}}}\cdot\vec{x}_{1}=f_{J_{\vec{x}_{1}%
}\cap\ J_{\vec{y}_{1}}}\cdot\vec{x}_{2}\ \text{in }\mathbb{A}^{m}(E)
\end{array}
\right\}  ,
\]
where$\ \mu_{1}=\frac{\alpha_{1}}{\alpha_{2}}.$

\subsubsection{The F-scalar case}

Let $[\varphi_{i}]_{\underline{\eta}^{i}}=$ diag$(\alpha_{i}\cdot\vec
{1},\alpha_{i}\cdot\vec{1})$ and $[g]_{\underline{\eta}^{i}}=\lambda
_{i}(g)\cdot$diag$(\vec{1},\vec{1})$ for some group homomorphisms $\lambda
_{i}:G\rightarrow GL_{2}(E),$ $i=1,2.$ Then
\[
(D_{1},\varphi_{1},L/K,E)\simeq(D_{2},\varphi_{2},L/K,E)
\]
if and only if $\alpha_{1}^{f}=\alpha_{2}^{f}$ and there exists some matrix
$R\in GL_{2}(E)\ $such that $\lambda_{2}(g)=R\lambda_{1}(g)R^{-1}$ for all $g$
and (with the notation of Section \ref{10})
\[
\left(  E^{\mid S_{L}\mid}\right)  \left(  \vec{x}_{1},\ f_{J_{\vec{y}_{1}}%
}\right)  \cdot\left(  R\cdot\text{diag}(\vec{1},\vec{1})\right)  =\left(
\ E^{\mid S_{L}\mid}\right)  \left(  \vec{x}_{2},\ f_{J_{\vec{y}_{2}}}\right)
.
\]

\subsubsection{The non-F-semisimple case}

Let
\[
\lbrack\varphi_{i}]_{\underline{\eta}^{i}}=\left(
\begin{array}
[c]{cc}%
\alpha_{i}\cdot\vec{1} & \vec{0}\\
\vec{1} & \alpha_{i}\cdot\vec{1}%
\end{array}
\right)  \ \text{with\ }\alpha_{i}\in E^{\times}\text{and\ }[g]_{\underline
{\eta}^{i}}=\text{diag}(\chi_{i}(g)\cdot\vec{1},\chi_{i}(g)\cdot\vec{1})
\]
for some characters $\chi_{i}:G\rightarrow E^{\times}.$ Then $(D_{1}%
,\varphi_{1},L/K,E)\simeq(D_{2},\varphi_{2},L/K,E)$ if and only if

\noindent(1) $\alpha_{1}^{f}=\alpha_{2}^{f}$ and in case that $\alpha_{1}%
\neq\alpha_{2},\ \sum\limits_{j=0}^{n(g)-1}\mu_{1}^{-2j}=0$ for all $g\in
G,\ $where$\ \mu_{1}=\frac{\alpha_{1}}{\alpha_{2}};$

\noindent\noindent(2)$\ \chi_{1}(g)=\mu_{1}^{n(g)}\cdot\chi_{2}(g)$ for all
$g\in G;$

\noindent(3)$\ J_{\vec{x}_{1}}=J_{\vec{x}_{2}}\ $and$\ $there exist $a\in
E^{\times}\ $and $c_{0}\in E\ $such that
\[
f_{J_{\vec{x}}}\cdot\vec{x}\cdot\vec{x}_{1}\cdot\vec{c}^{\otimes{e}}=\left(
f_{J_{\vec{x}}\cap J_{\vec{y}_{1}}}\cdot\vec{x}-f_{J_{\vec{x}}\cap J_{\vec
{y}_{1}}}\cdot\vec{x}_{1}\right)  \cdot\vec{a}^{\otimes{e}}\ \text{in
}\mathbb{A}^{m}(E),
\]
where $\vec{a}=a\cdot\left(  1,\mu_{1}^{-1},\mu_{1}^{-2},...,\mu_{1}%
^{-(f-1)}\right)  \ $and $\vec{c}=\left(  c_{0},c_{1},...,c_{f-1}\right)  $
with
\[
c_{i}=\mu_{1}^{i}\left\{  (c_{0}-a\mu_{1}^{-1}+a)-\sum\limits_{j=1}%
^{i-1}\left(  \mu_{1}^{-2j-1}-\mu_{1}^{-2j}\right)  \right\}  \ \text{for
}i=1,2,...,f-1.
\]

\section{Some consequences for crystalline
representations\label{crystalline example section}}

Let $K$ be any finite extension of $%
\mathbb{Q}
_{p}$ of absolute ramification index $e$ and absolute inertia degree $f.$ We
apply the results of the previous sections to study $2$-dimensional
crystalline $E$-representations of $G_{K}.$ Let $V$ be such a representation
and let $\left(  D,\varphi\right)  $ be the corresponding weakly admissible
filtered $\varphi$-module. Recall that the map $\varphi^{f}$ is $K_{0}\otimes
E$-linear. We call characteristic polynomial of $V\ $the characteristic
polynomial of $\varphi^{f},$ and throughout this section we assume that $V$ is
F-semisimple, meaning that $\varphi^{f}$ has the same property. Let
$\underline{\eta}\ $be a standard basis so that $[\varphi]_{\underline{\eta}%
}=$ diag$\left(  \alpha\cdot\vec{1},\delta\cdot\vec{1}\right)  $ with
$\alpha,\delta\in E^{\times}$ and $\alpha^{f}\neq\delta^{f},$ and let
\begin{equation}
\text{Fil}^{j}(D_{K})=\left\{
\begin{array}
[c]{l}%
\ \ \ \ \ \ \ \ \ \ \ \ D_{K}\text{ \ if\ \ \ }j\,\leq0,\\
\left(  E^{\mid S_{K}\mid_{I_{r}}}\right)  \left(  \vec{x}\cdot\eta_{1}%
+\vec{y}\cdot\eta_{2}\right)  \text{\ if}\\
1+w_{r-1}\leq j\leq w_{r},\ \text{for}\ r=0,...,t-1,\\
\ \ \ \ \ \ \ \ \ \ \ \ \ 0\text{\ \ if \ \ }j\geq1+w_{t-1}.
\end{array}
\right.  \label{filtrations in crystalline example}%
\end{equation}
for some vectors $\vec{x},\vec{y}\in E^{m},$ where $m$ is the degree of $K$
over $%
\mathbb{Q}
_{p},$ whose coordinates do not vanish simultaneously. In practice it is often
desirable to allow for a more flexible shape of Frobenius, at the cost of
adding extra rigidity to the filtrations. By Remark \ref{filtration} we may
assume that $\vec{y}=f_{J_{\vec{y}}},$ and by considering the ordered basis
$\underline{\zeta}=(\zeta_{1},\zeta_{2})$ with $\zeta_{1}=(\sum\limits_{i\in
J_{\vec{x}}^{^{\prime}}}e_{\tau_{i}}+\sum\limits_{i\in J_{\vec{x}}}x_{i}%
^{-1}e_{\tau_{i}})\eta_{1}$ and $\zeta_{2}=\eta_{2},$ we may further assume
that $\vec{x}=f_{J_{\vec{x}}}$ and $\vec{y}=f_{J_{\vec{y}}}.\ $In such a basis
the matrix of Frobenius remains diagonal of the form $[\varphi]_{\underline
{\zeta}}=$ diag$\left(  \vec{\alpha},\vec{\delta}\right)  $ for some vectors
$\vec{\alpha},\vec{\delta}\in\left(  E^{\times}\right)  ^{\mid S_{K_{0}}\mid}$
with $Nm_{\varphi}(\vec{\alpha})\neq Nm_{\varphi}(\vec{\delta}).$ The results
of Section \ref{filtrations classif potentially crystalline} take the form of
the following proposition.

\begin{proposition}
\label{last proposition}Let $(D_{i},\varphi_{i})$ be filtered $\varphi
$-modules with $[\varphi_{i}]_{\underline{\eta}^{i}}=$ diag$(\vec{\alpha}%
_{i},\vec{\delta}_{i}),$ $i=1,2$ and filtrations as in Section \ref{last list}%
, with $\vec{x}_{i}=f_{J_{\vec{x}_{i}}}$ and $\vec{y}_{i}=f_{J_{\vec{y}_{i}}%
},$ $i=1,2.$ The F-semisimple filtered $\varphi$-modules $(D_{i},\varphi_{i})$
are isomorphic if and only if either%
\[
\left\{
\begin{array}
[c]{c}%
Nm_{\varphi}(\vec{\alpha}_{1})=Nm_{\varphi}(\vec{\alpha}_{2}),\\
Nm_{\varphi}(\vec{\delta}_{1})=Nm_{\varphi}(\vec{\delta}_{2})
\end{array}
\right\}  ,\ \left\{
\begin{array}
[c]{c}%
J_{\vec{x}_{1}}=J_{\vec{x}_{2}},\\
J_{\vec{y}_{1}}=J_{\vec{y}_{2}}%
\end{array}
\right\}
\]
and $f_{J_{\vec{x}_{1}}\cap\ J_{\vec{y}_{1}}}\cdot\vec{a}=f_{J_{\vec{x}_{1}%
}\cap\ J_{\vec{y}_{1}}}\cdot\vec{d}\ $viewed in the projective space
$\mathbb{P}^{m-1}(E),$ \noindent where%
\[
\noindent\vec{a}=\left(  1,\frac{\alpha_{0}^{1}}{\alpha_{0}^{2}},\frac
{\alpha_{0}^{1}\alpha_{1}^{1}}{\alpha_{0}^{2}\alpha_{1}^{2}},\ldots
,\frac{\alpha_{0}^{1}\alpha_{1}^{1}\cdots\alpha_{f-2}^{1}}{\alpha_{0}%
^{2}\alpha_{1}^{2}\cdots\alpha_{f-2}^{2}}\right)  ^{\otimes e}\ \text{and}%
\noindent\ \vec{d}=\left(  1,\frac{\delta_{0}^{1}}{\delta_{0}^{2}}%
,\frac{\delta_{0}^{1}\delta_{1}^{1}}{\delta_{0}^{2}\delta_{1}^{2}}%
,\ldots,\frac{\delta_{0}^{1}\delta_{1}^{1}\cdots\delta_{f-2}^{1}}{\delta
_{0}^{2}\delta_{1}^{2}\cdots\delta_{f-2}^{2}}\right)  ^{\otimes e},
\]
or%
\[
\left\{
\begin{array}
[c]{c}%
Nm_{\varphi}(\vec{\alpha}_{1})=Nm_{\varphi}(\vec{\delta}_{2}),\\
Nm_{\varphi}(\vec{\delta}_{1})=Nm_{\varphi}(\vec{\alpha}_{2})
\end{array}
\right\}  ,\ \left\{
\begin{array}
[c]{c}%
J_{\vec{x}_{1}}=J_{\vec{y}_{2}},\\
J_{\vec{y}_{1}}=J_{\vec{x}_{2}}%
\end{array}
\right\}
\]
and $\noindent f_{J_{\vec{x}_{1}}\cap\ J_{\vec{y}_{1}}}\cdot\vec{b}%
=f_{J_{\vec{x}_{1}}\cap\ J_{\vec{y}_{1}}}\cdot\vec{c}\ $viewed in the
projective space $\mathbb{P}^{m-1}(E),\ $where%
\[
\noindent\vec{b}=\left(  1,\frac{\delta_{0}^{1}}{\alpha_{0}^{2}},\frac
{\delta_{0}^{1}\delta_{1}^{1}}{\alpha_{0}^{2}\alpha_{1}^{2}},\ldots
,\frac{\delta_{0}^{1}\delta_{1}^{1}\cdots\delta_{f-2}^{1}}{\alpha_{0}%
^{2}\alpha_{1}^{2}\cdots\alpha_{f-2}^{2}}\right)  ^{\otimes e}%
\ \text{and\noindent}\ \vec{c}=\left(  1,\frac{\alpha_{0}^{1}}{\delta_{0}^{2}%
},\frac{\alpha_{0}^{1}\alpha_{1}^{1}}{\delta_{0}^{2}\delta_{1}^{2}}%
,\ldots,\frac{\alpha_{0}^{1}\alpha_{1}^{1}\cdots\alpha_{f-2}^{1}}{\delta
_{0}^{2}\delta_{1}^{2}\cdots\delta_{f-2}^{2}}\right)  ^{\otimes e}.
\]
If all the $k_{i}\ $are $0,\ $any equation involving the sets $J_{\vec{x}_{i}%
},\ J_{\vec{y}_{i}}\ $should be ignored.$\ $
\end{proposition}

\noindent

The two cases of Proposition \ref{last proposition} occur due to the
isomorphism of any rank two filtered module which swaps its basis elements.
For our current normalization the results of Section
\ref{pot crystalline f semisimple} should be slightly modified: One should
only replace $efv_{p}(\alpha\delta)$ by $ev_{p}(Nm_{\varphi}(\vec{\alpha
})Nm_{\varphi}(\vec{\delta})),$ $efv_{p}(\alpha)$ by $ev_{p}(Nm_{\varphi}%
(\vec{\alpha}))$ and $ev_{p}(\delta)$ by $ev_{p}(Nm_{\varphi}(\vec{\delta})),$
where for a vector $\vec{a}$ we denote by $v_{p}(Nm_{\varphi}(\vec{a}))$ the
valuation of the product of its coordinates. For the rest of the section we
assume that our bases are standard with Frobenius as in Proposition
\ref{last proposition} and filtrations as in
(\ref{filtrations in crystalline example}) with $\vec{x}=f_{J_{\vec{x}}}$ and
$\vec{y}=f_{J_{\vec{y}}}.$ To avoid trivialities we assume that at least one
of the non negative weights $k_{i}$ is strictly positive. The following
corollary follows easily.

\begin{corollary}
\label{existance of inf families}Let $(D,\varphi)$ be an F-semisimple, weakly
admissible filtered $\varphi$-module of rank two over $K_{0}\otimes E $ with
labeled Hodge-Tate weights $(\{-k_{i},0\})_{\sigma_{i}}.$

$(1)$\ If $Tr(\varphi^{f})\in\mathcal{O}_{E}^{\times}$ then the corresponding
crystalline representation is reducible;

$(2)$ There exist infinite families of weakly admissible non isomorphic
F-semisimple rank two filtered $\varphi$-modules sharing the same
characteristic polynomial and filtration with $(D,\varphi)\ $if and only
if$\ \mid J_{\vec{x}}\cap J_{\vec{y}}\mid>1.\ $
\end{corollary}

\noindent Let $k:=\sum\limits_{i=0}^{m-1}k_{i},$ and let $\pi\in E^{\times}%
\ $be an $e$-th root of $p.\ $Let $\alpha\in m_{E}$ with $\alpha^{2}\neq
4\pi^{k}$ so that the roots $\varepsilon_{0},\varepsilon_{1}\ $of
$X^{2}-\alpha X+\pi^{k}$ be distinct. Consider the rank two filtered $\varphi
$-modules $D\left(  \vec{\lambda},\vec{\mu}\right)  ,$ with $\vec{\lambda
},\vec{\mu}\in\left(  E^{\times}\right)  ^{f-1},$ with Frobenius endomorphisms
given by
\[
\lbrack\varphi]_{\underline{\eta}}=\text{diag}\left(  \left(  \lambda
_{0},\lambda_{1},...,\lambda_{f-2},\frac{\varepsilon_{0}}{\lambda_{0}%
\lambda_{1}\cdots\lambda_{f-2}}\right)  ,\left(  \mu_{0},\mu_{1},...,\mu
_{f-2},\frac{\varepsilon_{1}}{\mu_{0}\mu_{1}\cdots\mu_{f-2}}\right)  \right)
,
\]
and filtrations as in (\ref{filtrations in crystalline example}) with $\vec
{x}=\vec{y}=\vec{1}.\ $We have the following corollary.

\begin{corollary}
\label{b}\label{call me something} $(1)$ For any $\vec{\lambda},\vec{\mu}%
\in\left(  E^{\times}\right)  ^{f-1},\ $the filtered modules $D\left(
\vec{\lambda},\vec{\mu}\right)  $ are irreducible and weakly admissible;

$(2)$\ $D\left(  \vec{\lambda},\vec{\mu}\right)  \simeq D\left(  \vec{\lambda
}_{1},\vec{\mu}_{1}\right)  $ if and only if $\vec{\lambda}\cdot\vec{\mu}%
_{1}=\vec{\lambda}_{1}\cdot\vec{\mu};$

$(3)$\ The filtered modules $D\left(  \vec{1},\vec{\mu}\right)  $ with
$\vec{\mu}\in\left(  E^{\times}\right)  ^{f-1}$ are representatives of the
distinct isomorphism classes of all rank two weakly admissible filtered
modules with fixed characteristic polynomial $X^{2}-\alpha X+\pi^{k}$ and
filtration as in (\ref{filtrations in crystalline example}), with $\vec
{x}=\vec{y}=\vec{1}.$
\end{corollary}

\begin{corollary}
\label{a}If $K\neq%
\mathbb{Q}
_{p}$ there exist (infinitely many) disjoint infinite families of irreducible
$2$-dimensional crystalline $E$-representations of $G_{K},$ sharing the same
characteristic polynomial and filtration.
\end{corollary}

\section*{Appendix\label{Appendix 1}}

\noindent\textbf{The potentially crystalline }$E^{\times}$\textbf{-valued
characters of \ }$G_{K}.\boldsymbol{\ }$Let $k_{0},k_{1},...,\noindent
k_{m-1}$ be arbitrary integers. Assume that there exists $\varpi\in E^{\times
}$ such that $\varpi^{em}=p^{\sum\limits_{i=0}^{m-1}k_{i}}.$ \noindent The
weakly admissible rank one filtered $(\varphi,L/K,E)$-modules with labeled
Hodge-Tate weights $(-k_{i})_{\sigma_{i}}$ are of the form $D=(\prod
\limits_{\ \ S_{L_{0}}}E)\eta$ with $\varphi(\eta)=u(\varpi,\varpi
,...,\varpi)\eta$ for some $\allowbreak u\in E^{\times}$ with $v_{p}(u)=0$
and, $g(\eta)=(\chi(g)\cdot\vec{1})\eta$ for some character

\noindent$\chi:$ Gal$(L/K)\rightarrow E^{\times}.$ Their filtrations are given
by
\[
\text{Fil}^{j}(D_{L})=\left\{
\begin{array}
[c]{l}%
\left(  E^{\mid S_{L}\mid}\right)  (1\otimes\eta)\ \text{ \ \ \ \ \ if\ }j\leq
w_{0},\\
\left(  E^{\mid S_{L}\mid_{I_{1}}}\right)  (1\otimes\eta)\ \ \ \ \text{ if
\ }1+w_{0}\leq j\leq w_{1},\\
\ \ \ \ \ \ \ \ \ \ \ \ \ \text{\ \ \ \ }\cdots\cdots\cdots\\
\left(  E^{\mid S_{L}\mid_{I_{t-1}}}\right)  (1\otimes\eta)\text{
\ if\ }1+w_{t-2}\leq j\leq w_{t-1},\\
\ \ \ \ 0\text{ \ \ \ \ \ \ \ \ \ \ \ \ \ \ \ \ \ \ \ \ \ if\ }j\geq1+w_{t-1},
\end{array}
\right.
\]
where the sets $I_{r}$ are unions of Gal$(L/K)$-orbits for all $r.$ Denote
such a filtered module by $(D_{u},\chi).$ Then $(D_{u},\chi)\simeq(D_{v}%
,\psi)$ if and only if \noindent(i) $u^{f}=v^{f}$ and (ii) $\chi
(g)=\varepsilon^{n(g)}\psi(g)$ for all $g\in G,$ where $\varepsilon=uv^{-1}.$

\section*{Acknowledgement}

\noindent Part of the paper was written during a visit at Max-Planck Institut
f\"{u}r Mathematik supported by MRTN-CT-2003-504917, AAG Network. I thank the
Max-Planck Institute for providing an ideal environment to work throughout my
stay there. I am indebted to the anonymous referee for detailed comments, for
pointing out minor mistakes, and for making numerous useful suggestions for
improvements which have been incorporated in the text.

\end{document}